\documentclass[11pt]{amsart}
\usepackage[latin1]{inputenc}
\usepackage[english]{babel}
\usepackage{amsmath}
\usepackage{amsthm, cite}
\usepackage{amssymb,stmaryrd}
\usepackage{amsthm, amsfonts, mathrsfs, amsfonts}
\usepackage{latexsym}
\usepackage{textcomp}
\usepackage{enumerate}
\usepackage[all]{xy}

\usepackage{xcolor}
\usepackage{hyperref}

\theoremstyle{definition}
\newtheorem{definition}{Definition}[section]
\newtheorem{notation}[definition]{Notation}
\newtheorem{remark}[definition]{Remark}
\newtheorem{example}[definition]{Example}

\theoremstyle{plain}
\newtheorem{cor}[definition]{Corollary}

\theoremstyle{plain}
\newtheorem{lemma}[definition]{Lemma}
\newtheorem{theorem}[definition]{Theorem}

\newtheorem{innercustomthm}{Theorem}
\newenvironment{customthm}[1]
 {\renewcommand\theinnercustomthm{#1}\innercustomthm}
 {\endinnercustomthm}

\renewcommand{\phi}{\varphi}

\newcommand{\One}{\operatorname{\mathbf{1}}}

\theoremstyle{definition}

\newcommand{\C}{\mathcal{C}}
\newcommand{\D}{\mathbf{D}}
\newcommand{\E}{\mathcal{E}}
\newcommand{\F}{\mathcal{F}}
\newcommand{\T}{\mathcal{T}}

\newcommand{\M}{\mathcal{M}}

\newcommand{\K}{\mathcal{K}}
\newcommand{\N}{\mathcal{N}}

\newcommand{\W}{\bold{W}}

\newcommand{\<}{\langle}
\renewcommand{\>}{\rangle}

\newcommand{\Aut}{\mathrm{Aut}}
\newcommand{\Hom}{\mathrm{Hom}}

\newcommand{\Mor}{\mathrm{Mor}}
\newcommand{\Out}{\mathrm{Out}}

\newcommand{\Syl}{\mathrm{Syl}}
\newcommand{\Iso}{\mathrm{Iso}}

\newcommand{\ob}{\mathrm{Ob}}

\newcommand{\id}{\operatorname{id}}
\newcommand{\typ}{{\operatorname{typ}}}

\newcommand{\tL}{{\widetilde{\mathcal{L}}}}
\newcommand{\tDelta}{\widetilde{\Delta}}
\newcommand{\tGamma}{\widetilde{\Gamma}}
\newcommand{\tS}{{\widetilde{S}}}

\newcommand{\tPi}{\widetilde{\Pi}}
\newcommand{\tD}{\widetilde{\mathbf{D}}}

\newcommand{\tf}{{\tilde{f}}}

\newcommand{\tF}{{\widetilde{\F}}}
\newcommand{\tT}{{\widetilde{\T}}}
\newcommand{\tdelta}{{\tilde{\delta}}}
\newcommand{\tpi}{{\tilde{\pi}}}
\newcommand{\tiota}{{\tilde{\iota}}}

\newcommand{\fN}{{\mathfrak{N}}}

\def \L {\mathcal L}
\def \H {\mathcal H}

\def \ov {\overline}

\title{Extensions of homomorphisms between localities}

\author[E.~Henke]{Ellen Henke}
\address{Institut f{\"u}r Algebra, Fakult{\"a}t Mathematik, Technische Universit{\"a}t Dresden, 01062 Dresden, Germany}
\email{ellen.henke@tu-dresden.de}
\thanks{The author would like to thank the Isaac Newton Institute for Mathematical Sciences, Cambridge, for support und hospitality during the programme Groups, representations and applications, where work on this paper was undertaken and supported by EPSRC grant no EP/R014604/1.}

\begin{document}

\maketitle

\begin{abstract}
We show that the automorphism group of a linking system associated to a saturated fusion system $\F$ depends only on $\F$ as long as the object set of the linking system is $\Aut(\F)$-invariant. This was known to be true for linking systems in Oliver's definition, but we demonstrate that the result holds also for linking systems in the considerably more general definition introduced previously by the author of this paper. A similar result is proved for linking localities, which are group-like structures corresponding to linking systems. Our argument builds on a general lemma about the existence of an extension of a homomorphism between localities. This lemma is also used to reprove a theorem of Chermak showing that there is a natural bijection between the sets of partial normal subgroups of two possibly different linking localities over the same fusion system. 
\end{abstract}

\section{Introduction}

Given a finite group $G$ with a Sylow $p$-subgroup $S$, the fusion system $\F_S(G)$ is the category whose objects are the subgroups of $S$ and whose morphisms are the injective group homomorphisms induced by conjugation in $G$. It turns out that the fusion system $\F_S(G)$ determines the $p$-completed classifying space $BG^\wedge_p$ up to homotopy; this statement is known as the Martino--Priddy conjecture and was first proved by Oliver \cite{Oliver2004,Oliver2006}. Fusion systems play also an important role in many other contexts, for example in a program announced by Aschbacher to revisit the classification of finite simple groups. The concept of a saturated fusion system generalizes the properties of $\F_S(G)$. In particular, a saturated fusion system is a category $\F$ which comes equipped with a $p$-group $S$ such that the objects of $\F$ are the subgroups of $S$ and the morphism sets consist of injective group homomorphisms subject to certain axioms. 

\smallskip

For the purposes of homotopy theory, Broto, Levi and Oliver \cite{BLO2} defined \emph{centric linking systems} associated to saturated fusion systems. A category $\L_S^c(G)$, which is a centric linking system associated to $\F_S(G)$, can be constructed directly from the group $G$. The $p$-completed classifying space $BG^\wedge_p$ of $G$ is homotopy equivalent to the $p$-completed nerve of the category $\L_S^c(G)$. This fact played an important role in the proof of the Martino-Priddy conjecture. In the abstract context, there is an essentially unique centric linking system associated to every saturated fusion system. This longstanding conjecture was proved by Chermak \cite{Ch} and subsequently by Oliver \cite{Oliver:2013}. Both proofs depend a priori on the classification of finite simple groups, but work of Glauberman--Lynd \cite{GL2016} removes the dependence of Oliver's proof on the classification.

\smallskip

Linking systems form not only the algebraic foundation for defining $p$-completed classifying spaces of fusion systems, but they are also important when studying extensions of fusion systems. The object set of a centric linking system associated to a fusion system $\F$ over $S$ is a certain set of subgroups of $S$ determined by $\F$. When studying extensions, one often wants to choose the object sets of linking systems more flexibly. At least partly for that reason, a more general notion of linking systems was introduced by Oliver \cite{OliverExtensions} (building on earlier work of Broto, Castellana, Grodal, Levi and Oliver \cite{controlling}). Linking systems are special cases of transporter systems as defined by Oliver and Ventura \cite{OV}. Extensions of linking systems and transporter systems were studied for example in \cite{AOV1}, \cite{BCGLO2007}, \cite{OV} and \cite{OliverExtensions}.

\smallskip

Chermak \cite{Ch} introduced with \emph{localities} group-like structures that correspond to transporter systems in a certain way. A locality consists more precisely of a ``partial group'' $\L$ (i.e. a set $\L$ with a ``product'' defined on some tuples of elements of $\L$ subject to group-like axioms), a ``Sylow $p$-subgroup'' $S$ of $\L$, and a set $\Delta$ of subgroups of $S$ (cf. Definitions~\ref{partial} and \ref{locality}). Here $\Delta$ is called the set of objects of $\L$ and turns out to be the object set of the transporter system corresponding to $(\L,\Delta,S)$. A rich theory of localities akin to the local theory of finite groups was developed by Chermak \cite{loc1,loc2,loc3}. Extensions of partial groups and localities were studied by Gonzalez \cite{Gonzalez} and are also the subject of work in progress of Valentina Grazian and the author of this paper. At least with the currently known conceptual framework, it seems in fact that there are some advantages to studying extensions of localities rather than extensions of linking systems or transporter systems. For example, for partial groups, there are natural notions of \emph{homomorphisms} and of \emph{partial normal subgroups} such that the kernels of the homomorphisms from a locality $\L$ are precisely the partial normal subgroups of $\L$. 

\smallskip

The author of this paper \cite{subcentric} suggested a definition of a linking system which is significantly more general than the previously existing notion, and this leads to the corresponding concept of a \emph{linking locality} (by Chermak \cite{loc2,loc3} also called a proper locality). 
It is one of the purposes of this paper to prove in this more general context some results which are known to hold for linking systems in Oliver's definition \cite[Definition~3]{OliverExtensions}. Another purpose of this paper is to prove a Lemma about homomorphisms between localities (Lemma~\ref{L:Main}) and to  reprove in Theorem~\ref{T:MainChermakII} a result of Chermak \cite[Theorem~A2]{loc2}. Both Lemma~\ref{L:Main} and Theorem~\ref{T:MainChermakII} are used in joint work of Chermak and the author of this paper \cite{normal} to show that there is a one-to-one correspondence between the normal subsystems of a fusion system and the partial normal subgroups of an associated linking locality. The theorems on linking localities proved in the present paper are also used in \cite{Henke:Regular}, where the results from \cite{loc3} are revisited and extended.

\smallskip

We will now explain our main results in more detail. When studying extensions of linking systems or linking localities, their automorphism groups play an important role. Thus, it is of interest to see that different linking systems or linking localities associated to the same fusion system $\F$ have the same automorphism group. This is indeed the case if we  consider linking systems and linking localities with $\Aut(\F)$-invariant object sets, as for example the typically used sets of $\F$-centric, $\F$-quasicentric or $\F$-subcentric subgroups. For linking localities, we prove the following theorem. In Theorem~\ref{T:MainIso} and Theorem~\ref{T:MainAut} we also prove some more general statements about isomorphisms and automorphisms of linking localities.

\begin{customthm}{A.1}\label{T:A1}
Let $\F$ be a saturated fusion system over $S$. If $(\L,\Delta,S)$ and $(\L^+,\Delta^+,S)$ are linking localities over $\F$ such that $\Delta$ and $\Delta^+$ are $\Aut(\F)$-invariant, then $\Aut(\L,\Delta,S)\cong \Aut(\L^+,\Delta^+,S)$. In the case that $\Delta\subseteq\Delta^+$ and $\L=\L^+|_\Delta$, a group isomorphism is given by
\[\Aut(\L^+,\Delta^+,S)\longrightarrow \Aut(\L,\Delta,S),\;\alpha\mapsto \alpha|_\L.\]
\end{customthm}

The reader is referred to Definition~\ref{D:RestrictionLocality} for the definition of the ``restriction'' $\L^+|_\Delta$. The above mentioned correspondence between transporter systems and localities (which we outline in Subsection~\ref{SS:TransLoc}) leads to a correspondence between linking systems and linking localities. Passing to the restriction $\L^+|_\Delta$ corresponds in the world of transporter systems to passing to the full subcategory with object set $\Delta$. Thus, we obtain the following theorem for linking systems.

\begin{customthm}{A.2}\label{T:A2}
Suppose $\F$ is a saturated fusion system. If $\T$ and $\T^+$ are linking systems associated to the same saturated fusion system $\F$ such that the object sets of $\T$ and $\T^+$ are  $\Aut(\F)$-invariant, then $\Aut(\T)\cong\Aut(\T^+)$. In the case that $\T$ is a full subcategory of $\T^+$, a group isomorphism $\Aut(\T^+)\longrightarrow \Aut(\T)$ is given by restriction.   
\end{customthm}

By $\Aut(\T)$ we mean here the group of isotypical self-equivalences of $\T$ which send inclusions to inclusions; see Definition~\ref{D:TransIso}. In the literature, $\Aut(\T)$ is often denoted by  $\Aut_{\operatorname{typ}}^I(\T)$. We emphasize also that the term linking system refers to a linking system in the general sense of \cite{subcentric} (cf. Definition~\ref{D:LinkingSystem}). A version of Theorem~\ref{T:A2} was proved before by Andersen, Oliver and Ventura \cite[Lemma~1.17]{AOV1} for linking systems in Oliver's definition, i.e. for linking systems whose objects are quasicentric subgroups. The precise statement is actually given for outer automorphism groups of linking systems. We formulate a similar result in Theorem~\ref{T:A2Outer}. For this purpose, we state in Lemma~\ref{L:ExactSequence} that, for any linking system $\T$ associated to a saturated fusion system $\F$, there is an exact sequence
\[1\longrightarrow Z(\F)\xrightarrow{\;\;\delta_S\;\;} \Aut_\T(S)\longrightarrow \Aut(\T)\longrightarrow \Out_{\typ}(\T)\longrightarrow 1.
\]
Again, this was known to be true for linking systems in Oliver's definition (cf. \cite[Lemma~1.14(a)]{AOV1}) and the proof of the more general statement is given by similar arguments.

\smallskip

Theorem~\ref{T:A2Outer} allows us to prove the following theorem from the corresponding statement for centric linking systems, which was shown by Broto, Levi and Oliver \cite[Theorem~8.1]{BLO2}. The statement was also known before for linking systems in Oliver's definition; see \cite[Theorem~4.22]{AKO}. For any space $X$, $\Out(X)$ denotes the group of homotopy classes of self-equivalences of $X$.

\begin{customthm}{B}\label{C:B}
Let $\T$ be a linking system associated to a saturated fusion system $\F$ such that $\ob(\T)$ is $\Aut(\F)$-invariant. Then there is an isomorphism
\[\Out_{\operatorname{typ}}(\T)\xrightarrow{\;\;\cong\;\;} \Out(|\T|^\wedge_p)\]
which sends the class of $\alpha\in\Aut(\T)$ to $|\alpha|^\wedge_p\colon |\T|^\wedge_p\rightarrow |\T|^\wedge_p$.
\end{customthm}

We show Theorem~\ref{T:A2} and some more general theorems about isomorphisms and automorphisms of linking systems (Theorems~\ref{T:MainIsoT} and \ref{T:MainAutTrans}) from the corresponding statements for linking localities via the one-to-one correspondence between localities and transporter systems. However, in Remark~\ref{R:AOV}, we outline how a direct proof could be given via similar arguments as in \cite[Lemma~1.17]{AOV1}. The crucial point in each of the proofs of Theorems~\ref{T:A1} and \ref{T:A2} is to show that the appropriate restriction map is surjective. The necessary argument for localities is similar to the argument for transporter systems in \cite[Lemma~1.17]{AOV1}, but it can be formulated in a very general way such that it becomes also useful in other contexts. Namely, in Lemma~\ref{L:Main} we show that, under certain assumptions, a homomorphism from a locality $(\L,\Delta,S)$ can be extended to a homomorphism from a locality $(\L^+,\Delta^+,S)$ with $\L^+|_\Delta=\L$. We use Lemma~\ref{L:Main} to give a new proof of \cite[Theorem~A2]{loc2} (stated as Theorem~\ref{T:MainChermakII}(a) below) in Section~\ref{S:PartialNormal}. Moreover, both Lemma~\ref{L:Main} and Theorem~\ref{T:MainChermakII} are used in  \cite{normal}. For any partial group $\L$, we denote by $\fN(\L)$ the set of partial normal subgroups of $\L$.

\begin{customthm}{C}\label{T:MainChermakII}
If $(\L,\Delta,S)$ and $(\L^+,\Delta^+,S)$ are linking localities over the same fusion system $\F$ with $\Delta\subseteq\Delta^+$ and $\L=\L^+|_\Delta$, then the following hold:
\begin{itemize}
\item [(a)] The map 
\[\Phi_{\L^+,\L}\colon\fN(\L^+)\longrightarrow \fN(\L),\;\N^+\mapsto \N^+\cap\L\]
is well-defined and bijective. Both $\Phi_{\L^+,\L}$ and $\Phi_{\L^+,\L}^{-1}$ are inclusion-preserving.
\item [(b)] If $\N^+\in\fN(\L^+)$ and $\N:=\N^+\cap\L\in\fN(\L)$ such that $\F_{S\cap\N}(\N)$ is $\F$-invariant, then $\F_{S\cap \N^+}(\N^+)=\F_{S\cap\N}(\N)$. 
\item [(c)] Let $\N^+,\K^+\in\fN(\L)$, set $\N:=\N^+\cap\L$, $\K:=\K^+\cap\L$ and $T:=\N^+\cap S=\N\cap S$. Then $\N=\K T$ if and only if $\N^+=\K^+ T$. 
\end{itemize}  
\end{customthm}

The statement in part (a) of the above theorem that $\Phi_{\L,\L^+}$ and its inverse are inclusion-preserving is equivalent to saying that every $\N^+\in\fN(\L^+)$ is the smallest partial normal subgroup of $\L^+$ containing $\N^+\cap\L$. As a corollary to Theorem~\ref{T:MainChermakII}(a) one can also show that any two linking localities over the same fusion system have the same number of partial normal subgroups; see Corollary~\ref{C:PartialNormal}. 

\smallskip

In his original proof of Theorem~\ref{T:MainChermakII}(a), Chermak goes into the (somewhat complicated) details of the construction of elementary expansions as introduced in \cite[Section~5]{Ch}. Applying Lemma~\ref{L:Main} makes this unnecessary in our new proof. We do however use \cite[Theorem~5.14]{Ch}, which is proved via elementary expansions. Theorem~\ref{T:MainChermakII}(c)  fills in a small gap in the proof of \cite[Lemma~7.3]{loc2}.

\smallskip

\emph{Organization of the paper.} After introducing some background in Section~\ref{S:LocFus}, we prove Lemma~\ref{L:Main}, which is used in the proofs of our main results. Theorems~\ref{T:A1} and \ref{T:A2} together with some more general theorems and with Theorem~\ref{C:B} are proved in Section~\ref{S:Iso}. In preparation for that, in Section~\ref{S:Trans}, we define automorphisms and isomorphisms of transporter systems (cf. Definition~\ref{D:TransIso}). Moreover, we explain the correspondence between localities and transporter systems, which is then used in Section~\ref{S:Iso} to prove theorems about linking systems from corresponding statements about linking localities. Finally, in Section~\ref{S:PartialNormal}, we prove Theorem~\ref{T:MainChermakII}. The proof of Theorem~\ref{T:MainChermakII} is independent of the results stated and proved in Sections~\ref{S:Trans} and \ref{S:Iso}.

\section{Localities and fusion systems}\label{S:LocFus}

In this section we will introduce some basic definitions and show some lemmas needed in the proofs of our main theorems. The reader is referred to \cite{AKO} for background on fusion systems and to \cite{Ch} and \cite{loc1} for a more comprehensive introduction to localities. We will however summarize the most important definitions and results concerning localities. In particular we will recall the definitions of homomorphisms, projections, isomorphisms and automorphisms of localities in Subsection~\ref{SS:LocalityHomomorphism}. Some background on morphisms of fusion systems is also provided in Subsection~\ref{SS:FusionMorphisms}.

\subsection{Partial groups} 
For any set $\M$, write $\W(\M)$ for the set of words in $\M$. If $u,v\in\W(\M)$, then $u \circ v$ denotes the concatenation of the two words. The empty word will be denoted by $\emptyset$.

\begin{definition}[Partial Group]\label{partial}
Let $\L$ be a non-empty set, let $\D$ be a subset of $\W(\L)$, let $\Pi \colon  \D \longrightarrow \L$ be a map and let $(-)^{-1} \colon \L \longrightarrow \L$ be an involutory bijection, which we extend to a map 
\[(-)^{-1} \colon \W(\L) \longrightarrow \W(\L), w = (g_1, \dots, g_k) \mapsto w^{-1} = (g_k^{-1}, \dots, g_1^{-1}).\]
We say that $\L$ is a \emph{partial group} with product $\Pi$ and inversion $(-)^{-1}$ if the following hold:
\begin{itemize}
\item[(PG1)]  $\L \subseteq \D$  (i.e. $\D$ contains all words of length 1), and
\[  u \circ v \in \D \Rightarrow u,v \in \D.\]
(So in particular, $\emptyset\in\D$.)
\item[(PG2)] $\Pi$ restricts to the identity map on $\L$;
\item[(PG3)] $u \circ v \circ w \in \D \Rightarrow u \circ (\Pi(v)) \circ w \in \D$, and $\Pi(u \circ v \circ w) = \Pi(u \circ (\Pi(v)) \circ w)$;
\item[(PG4)] $w \in  \D \Rightarrow  w^{-1} \circ w\in \D$ and $\Pi(w^{-1} \circ w) = \One$ where $\One:=\Pi(\emptyset)$.
\end{itemize}
\end{definition}

Note that any group $G$ can be regarded as a partial group with product defined on $\D=\W(G)$ by extending the ``binary'' product to a map 
\[\Pi_G\colon\W(G)\longrightarrow G,(g_1,g_2,\dots,g_n)\mapsto g_1g_2\cdots g_n.\] 
If $\L$ is a partial group with product $\Pi\colon\D\longrightarrow\L$ and $u=(f_1,f_2,\dots,f_n)\in\D$, then we write also $f_1f_2\cdots f_n$ for $\Pi(u)$.

\begin{lemma}\label{L:PartialGroup}
Let $\L$ be a partial group with product $\Pi\colon\D\longrightarrow\L$.
\begin{itemize}
 \item [(a)] If $u,v\in\W(\L)$ with $u\circ (\One)\circ v\in\D$, then $u\circ v\in\D$ and $\Pi(u\circ (\One)\circ v)=\Pi(u\circ v)$.
 \item [(b)] If $u,v,w\in\W(\L)$ such that $u\circ v\circ v^{-1}\circ w\in\D$, then $u\circ w\in\D$ and $\Pi(u\circ v\circ v^{-1}\circ w)=\Pi(u\circ w)$. 
\end{itemize}
\end{lemma}

\begin{proof}
 Let $u,v$ as in (a). If $u=v=\emptyset$, then by axiom (PG1) $u\circ v=\emptyset\in\D$, and by axiom (PG2) and the definition of $\One$, we have $\Pi(u\circ v)=\Pi(\emptyset)=\One=\Pi(\One)=\Pi(u\circ (\One)\circ v)$. So to prove (a), we may assume that $u\neq\emptyset$ or $v\neq \emptyset$.

\smallskip

For any element $f\in\L$, axiom (PG1) gives $f=(f)\circ\emptyset\in\D$. So by axioms (PG2) and (PG3) we have $(f,\One)=(f)\circ (\Pi(\emptyset))\in\D$ and $f=\Pi(f)=\Pi((f)\circ(\Pi(\emptyset)))=\Pi(f,\One)$. So if $u=(f_1,\dots,f_n)\neq \emptyset$, then $u\circ (\One)\circ v=(f_1,\dots,f_{n-1})\circ (f_n,\One)\circ v\in\D$ implies by axiom (PG3) that $u\circ v=(f_1,\dots,f_{n-1})\circ (\Pi(f_n,\One))\circ v\in\D$ and $\Pi(u\circ v)=\Pi(u\circ (\One)\circ v)$. So (a) holds in this case. A similar argument show (a) in the case that $v\neq\emptyset$.

\smallskip

For the proof of (b), let now $u,v,w\in\W(\L)$ be arbitrary such that $u\circ v\circ v^{-1}\circ w\in\D$. Then by axiom (PG3), we have $u\circ (\One)\circ w=u\circ (\Pi(v\circ v^{-1}))\circ w\in\D$ and $\Pi(u\circ v\circ v^{-1}\circ w)=\Pi(u\circ (\One)\circ w)$. Hence, (b) follows from (a). 
\end{proof}

\begin{definition}
Let $\L$ be a partial group with product $\Pi\colon\D\longrightarrow\L$. 
\begin{itemize}
\item For every $g\in \L$ we define
\[ \D(g) = \{ x\in \L \mid (g^{-1}, x, g) \in \D\}.\]
The map $c_g \colon \D(g) \longrightarrow \L$, $x \mapsto x^g = \Pi(g^{-1}, x, g)$ is the \emph{conjugation map} by $g$. 
\item If $\H$ is a subset of $\L$ and $\H \subseteq \D(g)$, then we set 
\[\H^g = \{ h^g \mid h \in \H\}.\]
\item If $P\subseteq \L$, then $N_\L(P)$ is the set of all $g\in\L$ such that $P\subseteq\D(g)$ and $P^g=P$. Similarly, if $P$ and $Q$ are subsets of $\L$, we write $N_\L(P,Q)$ for the set of all $g\in\L$ such that $P\subseteq\D(g)$ and $P^g\subseteq Q$. 
\item A \emph{partial subgroup} is a subset $\H\subseteq\L$ such that $h^{-1}\in\H$ for all $h\in\H$, and $\Pi(w)\in\H$ for all $w\in\D(\L)\cap\W(\H)$. A partial subgroup $\H$ of $\L$ is a called a \emph{subgroup} of $\L$ if $\W(\H)\subseteq\D(\L)$. 
\item If $\N$ is a partial subgroup of $\L$, then $\N$ is called a \emph{partial normal subgroup} if $n^f\in\N$ for all $f\in\L$ and all $n\in\N\cap\D(f)$. 
\end{itemize}
\end{definition}

We remark that a subgroup $\H$ of $\L$ is always a group in the usual sense with the group multiplication defined by $hg=\Pi(h,g)$ for all $h,g\in\H$. In particular, we can talk about $p$-subgroups of partial groups, meaning subgroups whose number of elements is a power of $p$.

\subsection{Localities}

Roughly speaking, localities are partial groups with some some extra structure, in particular with a ``Sylow $p$-subgroup'' and a set $\Delta$ of ``objects'' which in a sense determines the domain of the product. Crucial is the following definition.

\begin{definition}\label{D:Sf}
Let $\L$ be a partial group.
\begin{itemize}
\item If $\Delta$ is a collection of subgroups of $\L$, define $\D_\Delta$ to be the set of words $w=(g_1, \dots, g_k) \in \W(\L)$ such that there exist $P_0, \dots ,P_k \in \Delta$ with 
\[P_{i-1} \subseteq \D(g_i)\mbox{ and }P_{i-1}^{g_i} = P_i\mbox{ for all }1 \leq  i \leq k.\]
If such $P_0,\dots,P_k$ are given, then we say also that $w\in\D_\Delta$ via $P_0,P_1,\dots,P_k$, or just that $w\in\D_\Delta$ via $P_0$. In situations where we wish to stress the dependence of $\D_\Delta$ on $\L$ and on the product $\Pi\colon\D\longrightarrow\L$, we write $\D_\Delta(\L,\Pi)$ for $\D_\Delta$. 
\item Given a $p$-subgroup $S$ of $\L$ and $f\in\L$ set
\[S_f:=\{x\in S\colon x\in\D(f)\mbox{ and }x^f\in S\}.\]  
If we want to stress the dependence of $S_f$ on $\L$ and on the partial product and inversion on $\L$, then we write $S_f^\L$ for $S_f$.
\end{itemize}
\end{definition}

\begin{definition}\label{locality}
Let $\L$ be a finite partial group, let $S$ be a $p$-subgroup of $\L$ and let $\Delta$ be a non-empty set of subgroups of $S$. We say that $(\L, \Delta, S)$ is a \emph{locality} if the following hold:
\begin{enumerate}
\item $S$ is maximal with respect to inclusion among the $p$-subgroups of $\L$;
\item $\D = \D_\Delta$;
\item $\Delta$ is closed under taking $\L$-conjugates and overgroups in $S$; i.e. if $P \in \Delta$ then $P^g\in \Delta$ for every $g\in \L$ with $P \subseteq S_g$, and every subgroup of $S$ containing an element of $\Delta$ is an element of $\Delta$.
\end{enumerate}
\end{definition}

We remark that the above definition of a locality is a reformulation of the one given by Chermak \cite[Definition 2.8]{loc1}. As argued in \cite[Remark~5.2]{subcentric}, the two definitions are equivalent.

\begin{example}\label{E:LGammaM}
Let $M$ be a finite group, $S\in\Syl_p(M)$ and $\F=\F_S(M)$. Let $\Gamma$ be a non-empty $\F$-closed collection of subgroups of $S$. Set
\[\L_\Gamma(M):=\{g\in G\colon S\cap S^g\in \Gamma\}=\{g\in G\colon \mbox{There exists } P\in\Gamma\mbox{ with }P^g\leq S\}\]
and let $\D$ be the set of tuples $(g_1,\dots,g_n)\in\W(M)$ such that there exist $P_0,P_1,\dots,P_n\in\Gamma$ with $P_{i-1}^{g_i}=P_i$. Then $\L_\Gamma(M)$ forms a partial group whose product is the restriction of the multivariable product in $M$ to $\D$, and whose inversion map is the restriction of the inversion map on the group $M$ to $\L_\Gamma(M)$. Moreover, $(\L_\Gamma(M),\Gamma,S)$ forms a locality. See \cite[Example/Lemma~2.10]{Ch} for a proof. 
\end{example}

\begin{lemma}[Important properties of localities]\label{L:LocalitiesProp}
If $(\L,\Delta,S)$ is a locality, then the following hold:
\begin{itemize}
\item [(a)] $N_\L(P)$ is a subgroup of $\L$ for every $P\in\Delta$.
\item [(b)] Let $P\in\Delta$ and $g\in\L$ with $P\subseteq S_g$. Then $Q:=P^g\in\Delta$, $N_\L(P)\subseteq \D(g)$ and 
$$c_g\colon N_\L(P)\longrightarrow N_\L(Q)$$
is an isomorphism of groups. 
\item [(c)] Let $w=(g_1,\dots,g_n)\in\D$ via $(X_0,\dots,X_n)$. Then 
$$c_{g_1}\circ \dots \circ c_{g_n}=c_{\Pi(w)}$$
 is a group isomorphism $N_\L(X_0)\longrightarrow N_\L(X_n)$.
\item [(d)] For every $g\in\L$, $S_g\in\Delta$. In particular, $S_g$ is a subgroup of $S$. Moreover, $S_g^g=S_{g^{-1}}$ and $c_g\colon S_g\longrightarrow S$ is an injective group homomorphism.
\item [(e)] For every $g\in\L$, $c_g\colon \D(g)\longrightarrow \D(g^{-1})$ is a bijection with inverse map $c_{g^{-1}}$.
\item [(f)] For every $w\in\W(\L)$, $S_w$ is a subgroup of $S_{\Pi(w)}$, and $S_w\in\Delta$ if and only if $w\in\D$.
\end{itemize}
\end{lemma}

\begin{proof}
Properties (a),(b) and (c) correspond to the statements in \cite[Lemma~2.3]{loc1} except for the fact stated in (b) that $Q\in\Delta$, which is however clearly true if one uses our definition of a locality. Property (d) holds by \cite[Proposition~2.6(a),(b)]{loc1} and property (e) is stated in \cite[Lemma~2.5(c)]{Ch}. Property (f) is \cite[Corollary~2.7]{loc1}.
\end{proof}

\begin{lemma}\label{L:NLSbiset}
If $(\L,\Delta,S)$ is a locality, $r\in N_\L(S)$ and $f\in\L$, then $(r,f)$, $(f,r)$ and $(r^{-1},f,r)$ are words in $\D$. Moreover,
\[S_{(f,r)}=S_{fr}=S_f,\;S_{(r,f)}=S_{rf}=S_f^{r^{-1}}\mbox{ and }S_{f^r}=S_f^r.\]
\end{lemma}

\begin{proof}
We will use Lemma~\ref{L:LocalitiesProp}(f) frequently in this proof without further reference. As $S_f^r\leq S$, we have $S_f\subseteq S_{(f,r)}$. In particular, since $S_f\in\Delta$, we have $S_{(f,r)}\in\Delta$ and $(f,r)\in\D$. So by \cite[Lemma~1.4(d)]{loc1}, $(f,r,r^{-1})\in\D$ and $f=\Pi(f,r,r^{-1})=(fr)r^{-1}$. Applying the first stated property with $(fr,r^{-1})$ in place of $(f,r)$, we also get $S_{fr}\subseteq S_{(fr,r^{-1})}$. We see now that
\[S_f\subseteq S_{(f,r)}\subseteq S_{fr}\subseteq S_{(fr,r^{-1})}\subseteq S_{(fr,r^{-1})}\leq S_{(fr)r^{-1}}=S_f.\]
Hence, all the inclusions above are equalities and $S_f=S_{(f,r)}=S_{fr}$. 

\smallskip

Similarly, as conjugation by $r$ takes $S_f^{r^{-1}}\leq S$ to $S_f$, we have $S_f^{r^{-1}}\leq S_{(r,f)}\in\Delta$ and $(r,f)\in\D$. So by \cite[Lemma~1.4(d)]{loc1}, $(r^{-1},r,f)\in\D$ and $f=r^{-1}(rf)$. Similarly, $S_{rf}^r\leq S_{(r^{-1},rf)}\leq S_{r^{-1}(rf)}=S_f$ and thus $S_{rf}\leq S_f^{r^{-1}}$. Hence 
\[S_f^{r^{-1}}\subseteq S_{(r,f)}\subseteq S_{rf}\subseteq S_f^{r^{-1}}\]
and equality holds everywhere above, i.e. $S_f^{r^{-1}}=S_{(r,f)}=S_{rf}$. 

\smallskip

Note that $(r^{-1},f,r)\in\D$ via $S_f^r$. Using the properties proved above, we see now that $S_{f^r}=S_{(r^{-1}f)r}=S_{r^{-1}f}=S_f^r$.
\end{proof}

\subsection{Fusion systems of localities}

Similarly as we we can attach to a finite group a fusion system over a Sylow $p$-subgroup, we can attach a fusion system to a locality.

\begin{definition}
Let $(\L,\Delta,S)$ be a locality. 
\begin{itemize}
\item For all $P,Q\in\Delta$ set
\[\Hom_\L(P,Q):=\{c_g|_P\colon g\in N_\L(P,Q)\}.\]
\item We write $\F_S(\L)$ for the smallest fusion system over $S$ containing all the conjugation maps $c_f\colon S_f\longrightarrow S$ with $f\in\L$, or equivalently for the fusion system generated by the sets $\Hom_\L(P,Q)$, where $P,Q$ are elements of $\Delta$.
\item We say that $(\L,\Delta,S)$ is a locality over $\F$ to indicate that $\F=\F_S(\L)$.
\end{itemize}
\end{definition}

\begin{lemma}\label{L:LocalityoverFFnatural}
If $(\L,\Delta,S)$ is a locality over $\F$ and $P\in\Delta$, then the following hold:
\begin{itemize}
\item[(a)] For every morphism $\phi\in\Hom_\F(P,S)$, then there exists $f\in\L$ such that $P\leq S_f$ and $\phi=c_f|_P$.
\item[(b)] The subgroup $P$ is fully $\F$-normalized if and only if $N_S(P)\in\Syl_p(N_\L(P))$. Moreover, if so then for every $Q\in P^\F$, there exists $g\in N_\L(N_S(Q),S)$ such that $Q^g=P$.
\item[(c)] $N_\F(P)=\F_{N_S(P)}(N_\L(P))$.
\end{itemize}
\end{lemma}

\begin{proof}
For (a) see \cite[Lemma~5.6]{subcentric}. Part (c) follows easily from (a). For the proof of (b) assume first that $P$ is fully normalized. As $N_S(P)$ is a $p$-subgroup of $N_\L(P)$, we can pick a Sylow $p$-subgroup $T$ of $N_\L(P)$ such that $N_S(P)\leq T$. By \cite[Proposition~2.22(b)]{Ch}, there exists $x\in\L$ such that $T\subseteq\D(x)$ and $T^x\leq S$. Then in particular, $P\leq N_S(P)\leq S_x$ and by Lemma~\ref{L:LocalitiesProp}(b), $T^x\leq N_S(P^x)$. By (a), we have $P^x\in P^\F$ and thus, as $P$ is fully normalized, $|N_S(P^x)|\leq |N_S(P)|$. On the other hand, $|N_S(P)|\leq |T|=|T^x|\leq |N_S(P^x)|$. Hence equality holds and thus $N_S(P)=T$ is a Sylow $p$-subgroup of $N_\L(P)$.

\smallskip

Suppose now on the other hand that  $N_S(P)\in\Syl_p(N_\L(P))$. Take $Q\in P^\F$. By (a), there exists $f\in \L$ such that $Q^f=P$, and by Lemma~\ref{L:LocalitiesProp}(b), the map $c_f\colon N_\L(Q)\longrightarrow N_\L(P)$ is an isomorphism of groups. Hence, $N_S(Q)^f$ is a $p$-subgroup of $N_\L(P)$. As $N_S(P)\in\Syl_p(N_\L(P))$, by Sylow's theorem, there exists $a\in N_\L(P)$ such that $N_S(Q)^{fa}=(N_S(Q)^f)^a\leq N_S(P)$, where the equality uses Lemma~\ref{L:LocalitiesProp}(c). Then $g:=fa\in N_\L(N_S(Q),S)$ with $Q^g=(Q^f)^a=P^a=P$. Moreover, $|N_S(Q)|=|N_S(Q)^g|\leq |N_S(P)|$. Because $Q\in P^\F$ was arbitrary, this shows that $P$ is fully normalized. Hence, (b) holds.
\end{proof}

If $(\L,\Delta,S)$ is a locality and $\F=\F_S(\L)$, then notice that $\Delta$ is $\F$-closed as defined next. 

\begin{definition}
Let $\F$ be a fusion system over $S$, and let $\Delta$ be a set of subgroups of $S$. 
\begin{itemize}
\item The set $\Delta$ is \emph{closed under $\F$-conjugacy} if $P^\F\subseteq \Delta$ for every $P\in\Delta$.  
\item We call $\Delta$ \emph{$\F$-closed} if $\Delta$ is both closed under $\F$-conjugacy and overgroup closed in $S$.
\end{itemize}
\end{definition}

Important examples of $\F$-closed collections are the set $\F^c$ of $\F$-centric subgroups (cf. \cite[Definition~3.1]{AKO}), the set $\F^q$ of $\F$-quasicentric subgroups (cf. Definition~4.5 and Lemma~4.6(d) in \cite{AKO}) and the set $\F^s$ of subcentric subgroups (cf. Definition~1 and Proposition~3.3 in \cite{subcentric}).

\subsection{Morphisms of fusion systems}\label{SS:FusionMorphisms}

Throughout this subsection let $\F$ and $\tF$ be fusion systems over $S$ and $\tS$ respectively. 

\begin{definition}
We say that a group homomorphism $\alpha\colon S\longrightarrow \tS$ \textit{induces a morphism} from $\F$ to $\tF$ if, for each $\phi\in\Hom_\F(P,Q)$, there exists $\psi\in\Hom_{\tF}(P\alpha,Q\alpha)$ such that $(\alpha|_P)\psi=\phi(\alpha|_Q)$. 
\end{definition}

Note that, for any $\phi\in\Hom_\F(P,Q)$, a map $\psi\in\Hom_{\tF}(P\alpha,Q\alpha)$ as in the above definition is uniquely determined. So if $\alpha$ induces a morphism from $\F$ to $\tF$, then $\alpha$ induces a map 
\[\alpha_{P,Q}\colon\Hom_\F(P,Q)\longrightarrow \Hom_{\tF}(P\alpha,Q\alpha).\] 
Together with the map $P\mapsto P\alpha$ from the set of objects of $\F$ to the set of objects of $\tF$ this gives a functor from $\F$ to $\tF$. Moreover, $\alpha$ together with the maps $\alpha_{P,Q}$ ($P,Q\leq S$) is a morphism of fusion systems in the sense of \cite[Definition~II.2.2]{AKO}. We call $(\alpha,\alpha_{P,Q}\colon P,Q\leq S)$ the \textit{morphism induced by $\alpha$}.

\begin{definition}
Suppose $\alpha\colon S\longrightarrow \tS$ induces a morphism from $\F$ to $\tF$. We say that $\alpha$ \emph{induces an epimorphism} from $\F$ to $\tF$ if the induced morphism $(\alpha,\alpha_{P,Q}\colon P,Q\leq S)$ is a surjective morphism of fusion systems. This means that $\alpha$ is surjective as a map $S\longrightarrow \tS$ and $\F\alpha^*=\tF$, i.e. for all $P,Q\leq S$ with $\ker(\alpha)\leq P\cap Q$, the map $\alpha_{P,Q}$ is surjective. If $\alpha$ is in addition injective, then we say that $\alpha$ \emph{induces an isomorphism} from $\F$ to $\tF$. If $\alpha\in\Aut(S)$ and $\alpha$ induces a morphism from $\F$ to $\F$, then we say that $\alpha$ \emph{induces an automorphism} of $\F$. We will write $\Aut(\F)$ for the set of automorphisms of $S$ which induce an automorphism of $\F$. 
\end{definition}

If $\alpha\colon S\longrightarrow \tS$ is an isomorphism of groups, then it is easy to see that $\alpha$ induces an isomorphism from $\F$ to $\tF$ if and only if, for all $P,Q\leq S$ and every group homomorphism $\phi\colon P\longrightarrow Q$, 
\[\phi\in\Hom_\F(P,Q)\Longleftrightarrow \alpha^{-1}\phi\alpha\in\Hom_{\tF}(P\alpha,Q\alpha);\] 
if so then the map $\alpha_{P,Q}$ as above is given by $\phi\mapsto\alpha^{-1}\phi\alpha$. It follows from this observation that $\alpha$ induces an isomorphism from $\F$ to $\tF$ if and only if the inverse map $\alpha^{-1}$ induces an isomorphism from $\tF$ to $\F$.

\begin{lemma}\label{L:FusionEpi}
Suppose $\alpha\colon S\longrightarrow \tS$ induces an epimorphism from $\F$ to $\tF$. Let $\ker(\alpha)\leq R\leq S$. Then the following hold:
\begin{itemize}
 \item [(a)] $(R\alpha)^{\tF}=\{R_0\alpha\colon R_0\in R^\F\}$.
 \item [(b)] The subgroup $R$ is fully normalized if and only if $R\alpha$ is fully normalized.
 \item [(c)] The group homomorphism $\alpha|_{N_S(R)}\colon N_S(R)\longrightarrow N_{\tS}(R\alpha)$ induces an epimorphism from $N_\F(R)$ to $N_{\tF}(R\alpha)$.  
\end{itemize}
\end{lemma}

\begin{proof}
Property (a) is elementary to check, and property (b) follows from (a), since $N_S(R_0)\alpha=N_{\tS}(R_0\alpha)$ has order $|N_S(R_0)|/|\ker(\alpha)|$ for all $R_0\in R^\F$. 

\smallskip

For the proof of (c) let $P,Q\leq N_S(R)$ with $\ker(\alpha)\leq P\cap Q$, $\phi\in\Hom_\F(P,Q)$ and $\psi=\phi\alpha_{P,Q}\in\Hom_{\tF}(P\alpha,Q\alpha)$. We have then $\alpha|_P\psi=\phi\alpha|_Q$. Moreover, if $R\leq P$, then $\ker(\alpha)\leq R\phi$ as $\ker(\alpha)$ is strongly closed. Hence, $R\leq P\cap Q$ and $R\phi=R$ if and only if $R\alpha\leq P\alpha\cap Q\alpha$ and $(R\alpha)\psi=(R\phi)\alpha=R\alpha$. This implies (c).
\end{proof}

\begin{lemma}\label{L:GroupEpiFusionEpi}
Let $\alpha\colon G\longrightarrow \tilde{G}$ be an epimorphism from a group $G$ to a group $\tilde{G}$. Let $S\in\Syl_p(G)$ and $\tS=S\alpha\in\Syl_p(\tilde{G})$. Then $\alpha|_S$ induces an epimorphism from $\F_S(G)$ to $\F_{\tS}(\tilde{G})$.
\end{lemma}

\begin{proof}
Let $P,Q$ be subgroups of $S$. If $g\in G$ with $P^g\leq Q\leq S$, then $(P\alpha)^{g\alpha}=P^g\alpha\leq Q\alpha\leq\tS$ and $(\alpha|_P) (c_{g\alpha}|_{P\alpha})=(c_g|_P)(\alpha|_Q)$. So $\alpha|_S$ is fusion preserving and the corresponding morphism of fusion systems takes $c_g|_P$ to $c_{g\alpha}|_{P\alpha}$. To show that $\alpha|_S$ induces an epimorphism, assume now that $\ker(\alpha|_S)\leq P\cap Q$ and fix $h\in\tilde{G}$ with $(P\alpha)^h\leq Q\alpha$. Since $\alpha$ is an epimorphism, there exists $g\in G$ with $g\alpha=h$. We have then $P^g\alpha=(P\alpha)^h\leq Q\alpha$. As $\ker(\alpha|_S)=\ker(\alpha)\cap S\leq Q$, the group $Q$ is a Sylow $p$-subgroup of $\ker(\alpha)Q$, which is the preimage of $Q\alpha$ in $G$. Thus, by Sylow's theorem, there exists $n\in\ker(\alpha)$ with $P^{gn}\leq Q$. Replacing $g$ by $gn$, we may assume that $P^g\leq Q$. As seen at the beginning, this means that $c_h|_{P\alpha}\in\Hom_{\F_{\tS}(\tilde{G})}(P\alpha,Q\alpha)$ is the image of $c_g|_P\in\Hom_{\F_S(G)}(P,Q)$ under the morphism induced by $\alpha$.
\end{proof}

\subsection{Homomorphisms of partial groups.} \label{SS:LocalityHomomorphism} 
In this subsection, we will introduce natural notions of homomorphisms, projections, isomorphisms and automorphisms of partial groups and of localities. We state moreover a few simple results needed in the proofs of our main theorems.

\begin{notation}\label{N:PartialHomWordMap}
If $\L$ and $\tL$ are sets and $\alpha\colon\L\longrightarrow\tL,f\mapsto f\alpha$ is a map, then we denote by $\alpha^*$ the induced map on words 
\[\W(\L) \longrightarrow \W(\tL),\quad w=(f_1,\dots,f_n)\mapsto w\alpha^*=(f_1\alpha,\dots,f_n\alpha).\]
If $\D\subseteq\W(\L)$, set $\D\alpha^*:=\{w\alpha^*\colon w\in\D\}$. 
\end{notation}

For the remainder of this subsection let $\L$ and $\tL$ be partial groups with products $\Pi\colon\D\longrightarrow\L$ and $\tPi\colon\tD\longrightarrow\tL$ respectively.

\begin{definition}\label{D:PartialHom}
A map $\alpha\colon\L\longrightarrow\tL$ is called a \emph{homomorphism of partial groups} if 
\begin{enumerate}
\item $\D\alpha^* \subseteq \tD$; and 
\item $\Pi(w)\alpha = \tPi(w\alpha^*)$ for every $w \in \D$.
\end{enumerate}
If moreover $\D\alpha^* = \tD$, then we say that $\alpha$ is a \emph{projection} of partial groups. If $\alpha$ is injective and $\D\alpha^*=\tD$, then $\alpha$ is called an \emph{isomorphism}. The isomorphisms of partial groups from $\L$ to itself are called \emph{automorphisms} and the set of these automorphisms is denoted by $\Aut(\L)$. 

For any homomorphism $\alpha\colon\L\longrightarrow\tL$, we call $\ker(\alpha)=\{f\in\L\colon f\alpha=\One\}$ the \emph{kernel} of $\alpha$. 
\end{definition}

Notice that every projection $\L\longrightarrow\tL$ is surjective, as $\tD$ contains all the words of length one. In particular, every isomorphism is a bijection. In fact, there is the following characterization of isomorphisms.

\begin{lemma}\label{L:PartialIso}
A map $\alpha\colon \L\longrightarrow\tL$ is an isomorphism of partial groups if and only if $\alpha$ is bijective and $\alpha$ and $\alpha^{-1}$ are both homomorphisms of partial groups.
\end{lemma}

\begin{proof}
If $\alpha$ is bijective and $\alpha$ and $\alpha^{-1}$ are both homomorphisms of partial groups, then $\D\alpha^*\subseteq\tD$ and $\tD(\alpha^{-1})^*\subseteq \D$, with the latter inclusion implying $\tD\subseteq\D\alpha^*$. Thus, we get $\D\alpha^*=\tD$. As $\alpha$ is an injective homomorphism of partial groups, this yields that $\alpha$ is an isomorphism of partial groups.

\smallskip

Assume now that $\alpha$ is an isomorphism of partial groups. Then $\alpha$ is a bijection. Moreover, $\D\alpha^*=\tD$ and thus $\tD(\alpha^{-1})^*=\D$. Given $w\in\tD$, it remains to show that $\tPi(w)\alpha^{-1}=\Pi(w(\alpha^{-1})^*)$. Note that $w(\alpha^{-1})^*\in\D$ and thus, as $\alpha$ is a homomorphism of partial groups, $\Pi(w(\alpha^{-1})^*)\alpha=\tPi(w(\alpha^{-1})^*\alpha^*)=\tPi(w)$. This implies the required equality.
\end{proof}

\begin{lemma}\label{L:PartialProj}
Suppose $\alpha\colon \L\longrightarrow \tL$ is a homomorphism of partial groups. If $M$ is a subgroup of $\L$, then $M\alpha$ is a subgroup of $\tL$ and $\alpha$ restricts to a group homomorphism $M\longrightarrow M\alpha$. 
\end{lemma}

\begin{proof}
If $w=(f_1,\dots,f_n)\in \W(M\alpha)$, then for $i=1,\dots,n$, there exists $g_i\in M$ such that $f_i=g_i\alpha$. It follows $u:=(g_1,\dots,g_n)\in\W(M)\subseteq\D$ and $w=u\alpha^*\in\tD$. Moreover, $\tPi(w)=\tPi(u\alpha^*)=\Pi(u)\alpha\in M\alpha$ as $M$ is a subgroup. Hence, $M\alpha$ is a subgroup of $\tL$. The assertion follows since $(gh)\alpha=\Pi(v)\alpha=\tPi(v\alpha^*)=(g\alpha)(h\alpha)$ for every word $v=(g,h)\in\W(M)$ of length two.
\end{proof}

We now turn attention to maps between localities.

\begin{definition}\label{D:LocalityHomomorphism}
Let $(\L, \Delta, S)$ and $(\tL, \tDelta,\tS)$ be localities and let $\alpha \colon \L \longrightarrow \tL$ be a projection of partial groups. 
\begin{itemize}
\item For any set $\Gamma$ of subgroups of $\L$, set 
\[\Gamma\alpha:=\{P\alpha\colon P\in\Gamma\}.\]
\item We say that $\alpha$ is a \emph{projection of localities} from $(\L,\Delta,S)$ to $(\tL,\tDelta,\tS)$ if $\Delta\alpha= \tDelta$. 
\item If $\alpha$ is a projection of localities which is injective (and thus an isomorphism of partial groups), then $\alpha$ is a called an \emph{isomorphism} of localities. We write $\Iso((\L,\Delta,S),(\tL,\tDelta,\tS))$ for the set of isomorphisms from $(\L,\Delta,S)$ to $(\tL,\tDelta,\tS)$ (which may be empty). 
\item Given a set $\Gamma$ of subgroups of $S$ and a set $\tGamma$ of subgroups of $\tS$, we write 
\[\Iso((\L,\Delta,S),(\tL,\tDelta,\tS))_{\Gamma,\tGamma}\]
for the set of isomorphisms $\alpha$ from $(\L,\Delta,S)$ to $(\tL,\tDelta,\tS)$ with $\Gamma\alpha=\tGamma$.
\item An isomorphism from $(\L,\Delta,S)$ to itself is called an \emph{automorphism}. We write $\Aut(\L,\Delta,S)$ for the group of automorphisms of $(\L,\Delta,S)$. If $\Gamma$ is a set of subgroups of $S$, then $\Aut(\L,\Delta,S)_{\Gamma}$ denotes the set of automorphisms $\alpha$ of $(\L,\Delta,S)$ with $\Gamma\alpha=\Gamma$. 
\item An automorphism of $(\L,\Delta,S)$ is called \emph{rigid}, if it restricts to the identity on $S$. 
\end{itemize}
\end{definition}

If $\alpha$ is a projection of localities from $(\L,\Delta,S)$ to $(\tL,\tDelta,\tS)$, then notice that $\alpha$ maps $S$ to $\tS$, as $S$ and $\tS$ are the unique maximal elements of $\Delta$ and $\tDelta$ respectively. In particular, $\Aut(\L,\Delta,S)$ acts on $S$ for every locality $(\L,\Delta,S)$.

\begin{lemma}\label{L:ProjectionLocalityMorphismFusionSystem}
Suppose $\alpha\colon\L\longrightarrow\tL$ is a projection from a locality $(\L,\Delta,S)$ to a locality $(\tL,\tDelta,\tS)$. Then the following hold: 
\begin{itemize}
\item [(a)] $N_\L(R)\alpha=N_{\tL}(R\alpha)$ for every $R\leq S$ with $S\cap\ker(\alpha)\leq R$.  
\item [(b)] The map $\alpha|_S\colon S\longrightarrow\tS$ induces an epimorphism of fusion systems from $\F_S(\L)$ to $\F_{\tS}(\tL)$.
\item [(c)] If $\alpha$ is an isomorphism, then $\tS_{f\alpha}=S_f\alpha$ for every $f\in\L$.
\end{itemize}
\end{lemma}

\begin{proof}
For the proof of (a) let $T:=S\cap\ker(\alpha)\leq R\leq S$. By \cite[Lemma~3.1(a)]{loc1}, $T$ is strongly closed in $\F_S(\L)$. Clearly, $N_\L(R)\alpha\subseteq N_{\tL}(R\alpha)$. Let $\tf\in N_{\tL}(R\alpha)$ and write $P$ for the full preimage of $\tS_{\tf}$ in $S$. Then $T\leq R\leq P$ and $\tf\in N_{\tL}(P\alpha,\tS)$. Hence, by \cite[Theorem~4.3(c)]{loc1}, we may choose $f\in N_\L(P,S)$ with $f\alpha=\tf$. Then $R^f\leq S$ and  $R^f\alpha=(R\alpha)^{\tf}=R\alpha$. So $R^f=R$ as $T=T^f\leq R\cap R^f$. Hence, we have shown that $f\in N_\L(R)$ and thus that $N_{\ov{\L}}(\ov{R})\subseteq N_\L(R)\alpha$. This proves (a).

\smallskip

The fusion system $\F_S(\L)$ is generated by maps of the form $c_f\colon P\longrightarrow Q$, where $P,Q\in\Delta$ and $f\in N_\L(P,Q)$. Similarly, $\F_{\tS}(\tL)$ is generated by maps of the form $c_{\tf}\colon P\alpha\longrightarrow Q\alpha$ where $P,Q\in\Delta$ and $\tf\in N_{\tL}(P\alpha,Q\alpha)$. Fixing $P,Q\in\Delta$, by \cite[Theorem~4.3(c)]{loc1}, $\alpha$ induces a surjection $N_\L(P,Q)\longrightarrow N_{\tL}(P\alpha,Q\alpha)$. Moreover, if $f\in N_\L(P,Q)$, then $(c_f|_P)(\alpha|_Q)=(\alpha|_P)(c_{f\alpha}|_{P\alpha})$. This implies (b). 

\smallskip

For the proof of (c) let $f\in\L$ be arbitrary and suppose $\alpha$ is an isomorphism. Using that $\alpha$ maps $S$ isomorphically to $\tS$ and that  $(f^{-1})\alpha=(f\alpha)^{-1}$ by \cite[Lemma~1.13]{loc1}, we see
\begin{eqnarray*}
S_f\alpha&=&\{s\alpha\colon s\in S,\;(f^{-1},s,f)\in\D,\;s^f\in S\}\\ 
&=&\{s\alpha\colon s\in S,\;((f\alpha)^{-1},s\alpha,f\alpha)\in\tD,\;(s\alpha)^{f\alpha}\in \tS\}\\ 
&=&\{t\in \tS\colon ((f\alpha)^{-1},t,f\alpha)\in\tD,\;t^{f\alpha}\in \tS\}\\ 
&=& (\tS)_{f\alpha}.
\end{eqnarray*}
\end{proof}

\subsection{Restrictions of localities}\label{SS:Restrictions}

\begin{definition}\label{D:RestrictionLocality}
Let $(\L^+,\Delta^+,S)$ be a locality with partial product $\Pi^+\colon\D^+\longrightarrow\L^+$, and let $\Delta\subseteq \Delta^+$ be closed with respect to taking $\L^+$-conjugates and overgroups in $S$. Suppose $\Delta$ is non-empty. Then we set 
\[\L^+|_{\Delta}:=\{f\in\L^+\colon S_f\in\Delta\}.\]
Note that $\D:=\D_{\Delta}(\L^+,\Pi^+)\subseteq\D^+\cap\W(\L^+|_{\Delta})$ and, by Lemma~\ref{L:LocalitiesProp}(c), $\Pi^+(w)\in\L|_{\Delta}$ for all $w\in\D$. We call $\L:=\L^+|_{\Delta}$ together with the partial product $\Pi^+|_{\D}\colon\D\longrightarrow \L$ and the restriction of the inversion map on $\L^+$ to $\L$ the \emph{restriction} of $\L^+$ to $\Delta$.
\end{definition}

The properties of the restriction $\L^+|_\Delta$ are summarized in the following lemma, which we will use throughout, most of the time without reference.

\begin{lemma}\label{L:RestrictionProp}
Let $(\L^+,\Delta^+,S)$ be a locality with partial product $\Pi^+\colon\D^+\longrightarrow\L^+$, and let $\Delta\subseteq \Delta^+$ be non-empty and closed with respect to taking $\L^+$-conjugates and overgroups in $S$. Set $\L:=\L^+|_\Delta$, $\D:=\D_\Delta(\L^+,\Pi^+)$ and $\Pi:=\Pi^+|_\D$. 
\begin{itemize}
 \item [(a)] $\L$ together with $\Pi\colon\D\longrightarrow \L$ and the restriction of the inversion map on $\L^+$ to $\L$ forms a partial group.
 \item [(b)] If $f\in\L$, then it does not matter whether we form $S_f$ inside of $\L$ or of $\L^+$, i.e. $S_f^\L=S_f^{\L^+}$ (with the notation introduced in Definition~\ref{D:Sf}).
 \item [(c)] The triple $(\L,\Delta,S)$ is a locality. 
\end{itemize}
\end{lemma}

\begin{proof}
Part (a) is straightforward to check. Let $f\in\L$. As $\D\subseteq\D^+$ and $\Pi=\Pi^+|_\D$, clearly $S_f^\L\subseteq S_f^{\L^+}$. Setting $P:=S_f^{\L^+}$, by definition of $\L|_\Delta$, we have $P\in\Delta$. Moreover, the conjugate $P^f$ is defined in $\L^+$ and an element of $\Delta$, as $\Delta$ is closed under taking $\L$-conjugates. Now for $a\in P$, we have $(f^{-1},a,f)\in\D=\D_\Delta(\L^+,\Pi^+)$ via $P^f,P,P,P^f$. Hence, $P^f$ is defined in $\L$, which implies $P\subseteq S_f^{\L}$. This shows (b). 

\smallskip

The proof of (c) is given in \cite[Lemma~2.21(a)]{Ch}, but we repeat the argument here in detail, since we feel that there is a small gap in the proof: Note that $S\in\Delta$ and so $\W(S)\subseteq \W(N_\L(S))\subseteq \D=\D_\Delta(\L^+,\Pi^+)$. Hence, $S$ is a $p$-subgroup of $\L$. As $\D\subseteq\D^+$ and $\Pi=\Pi^+|_\D$, every $p$-subgroup of $\L$ is also a $p$-subgroup of $\L^+$. Therefore, $S$ is a maximal $p$-subgroup of $\L$, since it is a maximal $p$-subgroup of $\L^+$. By assumption, $\Delta$ is closed under taking $\L^+$-conjugates and overgroups in $S$, so it is in particular closed under taking $\L$-conjugates in $S$. Thus, it remains to show that $\D_\Delta(\L,\Pi)=\D$. Clearly, $\D_\Delta(\L,\Pi)\subseteq\D:=\D_\Delta(\L^+,\Pi^+)$. If $w=(f_1,\dots,f_n)\in\D:=\D_\Delta(\L^+,\Pi^+)$ via $P_0,\dots,P_n\in\Delta$, this means that the conjugate $P_{i-1}^{f_i}$ is defined in $\L^+$ and equal to $P_i$ for $i=1,\dots,n$. Then $P_{i-1}\subseteq S_{f_i}^{\L^+}=S_{f_i}^\L$ by (b). Hence, $w\in\D_\Delta(\L,\Pi)$ via $P_0,P_1,\dots,P_n$. This proves (c). 
\end{proof}

\begin{lemma}\label{L:IsoRestrictionLocality}
Let $(\L^+,\Delta^+,S)$ and $(\tL^+,\tDelta^+,\tS)$ be localities. Let $\emptyset\neq\Delta\subseteq\Delta^+$ and $\emptyset\neq\tDelta\subseteq\tDelta^+$ such that $\Delta$ is  closed under taking $\L^+$-conjugates and overgroups in $S$ and $\tDelta$ is closed under taking $\tL$-conjugates and overgroups in $\tS$. Set $\L:=\L^+|_{\Delta}$ and $\tL:=\tL^+|_{\tDelta}$. Then $\gamma|_{\L}\in\Iso((\L,\Delta,S),(\tL,\tDelta,\tS))_{\Delta^+,\tDelta^+}$ for every $\gamma\in\Iso((\L^+,\Delta^+,S),(\tL^+,\tDelta^+,\tS))_{\Delta,\tDelta}$. 
\end{lemma}

\begin{proof}
If $f\in\L^+$, then by Lemma~\ref{L:ProjectionLocalityMorphismFusionSystem}(c), $\tS_{f\gamma}=S_f\gamma$. As $\Delta\gamma=\tDelta$ and $\gamma$ is bijective, this means $S_f\in\Delta$ if and only if $\tS_{f\gamma}\in\tDelta$. Hence, $f\in\L$ if and only if $f\gamma\in\tL$, i.e. $\gamma|_{\L}\colon\L\longrightarrow\tL$ is well-defined and surjective. Clearly, $\gamma|_\L$ is injective.

\smallskip

Write $\Pi\colon\D\longrightarrow \L$ and $\tPi\colon \tD\rightarrow\tL$ for the products on $\L$ and $\tL$ respectively. Let $w=(f_1,\dots,f_n)\in\D$ via $P_0,\dots,P_n\in\Delta$, i.e. $P_{i-1}\leq S_{f_i}$ and $P_{i-1}^{f_i}=P_i$ for $i=1,\dots,n$. Then $P_{i-1}\gamma\leq (S_{f_i})\gamma=\tS_{f_i\gamma}$ and, as $\gamma$ is a homomorphism of partial groups,  $(P_{i-1}\gamma)^{f_i\gamma}=(P_{i-1}^{f_i})\gamma=P_i\gamma$. Since $\Delta\gamma=\tDelta$, this shows that $w\gamma^*=(f_1\gamma,\dots,f_n\gamma)\in\tD$ via $P_0\gamma,\dots,P_n\gamma\in\tDelta$. Hence, $\D\gamma^*\subseteq\tD$. As $\gamma^{-1}$ is an isomorphism from $(\tL^+,\tDelta^+,S)$ to $(\L^+,\Delta^+,S)$ by Lemma~\ref{L:PartialIso}, a symmetric argument shows that $\tD(\gamma^{-1})^*\subseteq\D$ and thus $\tD\subseteq\D\gamma^*$. This proves $\D\gamma^*=\tD$. As $\gamma\colon\L^+\longrightarrow\tL^+$ is a homomorphism of partial groups and since $\Pi$ and $\tPi$ are restrictions of the products on $\L^+$ and $\tL^+$ respectively, we have   $\tPi(w\gamma^*)=\Pi(w)\gamma$ for all $w\in\D$. So $\gamma|_{\L}$ is an isomorphism of partial groups from $\L$ to $\tL$ and the assertion follows.
\end{proof}

\subsection{Linking localities}\label{SS:LinkingLoc}

\begin{definition}\label{D:LinkingLoc}~
\begin{itemize}
\item A finite group $G$ is said to be of \emph{characteristic $p$} of $C_G(O_p(G))\leq O_p(G)$. 
\item A locality $(\L,\Delta,S)$ is called a \emph{linking locality} if $\F_S(\L)$ is saturated, $N_\L(P)$ is of characteristic $p$ for every $P \in \Delta$, and $\F_S(\L)^{cr}\subseteq \Delta$.
\item If $\F$ is a saturated fusion system over a $p$-group $S$, then a subgroup $P\leq S$ is called \emph{$\F$-subcentric} if $N_\F(Q)$ is constrained for every fully $\F$-normalized $\F$-conjugate $Q$ of $P$. Write $\F^s$ for the set of $\F$-subcentric subgroups of $S$. 
\item A \emph{subcentric linking locality} is a linking locality $(\L,\Delta,S)$ such that $\Delta=\F_S(\L)^s$.
\end{itemize}  
\end{definition}

Linking localities are closely related to linking systems. We provide some more details on that in Subsection~\ref{SS:LinkingSystems}. Given a saturated fusion system $\F$, it is elementary to show that the object set $\Delta$ of a linking locality over $\F$ is always contained in $\F^s$. On the other hand, using the existence and uniqueness of centric linking systems, it is shown in \cite[Theorem~A]{subcentric} that, for every $\F$-closed set $\Delta$ with $\F^{cr}\subseteq\Delta\subseteq\F^s$, there exists a linking locality $(\L,\Delta,S)$ over $\F$ which is unique up to rigid isomorphism. Moreover, it is proved that the set $\F^s$ is $\F$-closed and thus there exists a subcentric linking locality over $\F$ which is unique up to rigid isomorphism. 

\smallskip

We will need the following slightly technical lemma.

\begin{lemma}\label{L:LinkingLocN(R)}
Suppose $(\L,\Delta,S)$ and $(\L^+,\Delta^+,S)$ are linking localities over the same fusion system $\F$ such that $\Delta\subseteq\Delta^+$ and $\L=\L^+|_\Delta$. Let $R\in\Delta^+\backslash\Delta$ such that $R$ is fully normalized and every proper overgroup of $R$ is in $\Delta$. Then $N_\L(R)=N_{\L^+}(R)$ is a subgroup of $R$. Moreover, $R^*=O_p(N_{\L^+}(R))\in\Delta$ and $N_\L(R)=N_{N_\L(R^*)}(R)$. 
\end{lemma}

\begin{proof}
As $R\not\in\Delta$ and $\F^{cr}\subseteq\Delta$, we have $R\not\in\F^{cr}$. By \cite[Lemma~6.2]{subcentric}, this implies $R<R^*:=O_p(N_{\L^+}(R))$ and so $R^*\in\Delta$. Hence, using Lemma~\ref{L:RestrictionProp}(b), we see that $N_{\L^+}(R)\subseteq N_{\L^+}(R^*)=N_\L(R^*)$  and $N_{\L^+}(R)=N_{N_{\L^+}(R^*)}(R)=N_{N_\L(R^*)}(R)=N_\L(R)$ is a subgroup of $\L$.
\end{proof}

\section{A crucial lemma}

In this section we prove the following general lemma on which the proofs of Theorem~\ref{T:A1} and Theorem~\ref{T:MainChermakII} will be based on. It is also used in \cite{normal} to show that there is a one-to-one correspondence between the normal subsystems of a saturated fusion system and the partial normal subgroups of an associated linking locality.

\begin{lemma}\label{L:Main}
Let $(\L^+,\Delta^+,S)$ and $(\tL,\tDelta,\tS)$ be localities, set $\F:=\F_S(\L^+)$ and let $\Delta$ be a non-empty subset of $\Delta^+$ which is $\F$-closed. Set $\L:=\L^+|_\Delta$. Let $T$ be a strongly $\F$-closed subgroup such that
\[\Gamma^+:=\{P\cap T\colon P\in\Delta^+\}\subseteq\Delta^+\mbox{ and }\Gamma:=\{P\cap T\colon P\in\Delta\}\subseteq\Delta.\]
Assume that every proper overgroup in $T$ of an element of $\Gamma^+\backslash \Gamma$ is in $\Gamma$. Suppose we are given
\begin{itemize}
 \item a homomorphism of partial groups $\alpha\colon \L\longrightarrow \tL$ with $\Delta^+\alpha\subseteq\tDelta$; 
 \item a set $\Gamma_0\subseteq\Gamma^+\backslash\Gamma$ of fully $\F$-normalized representatives of the $\F$-conjugacy classes of the subgroups in $\Gamma^+\backslash\Gamma$; and
 \item for each $Q\in\Gamma_0$ a homomorphism of groups $\alpha_Q\colon N_{\L^+}(Q)\longrightarrow N_{\tL}(Q\alpha)$ with $\alpha_Q|_{N_{\L}(Q)}=\alpha|_{N_\L(Q)}$.
\end{itemize}
Then there exists a unique homomorphism of partial groups $\gamma\colon \L^+\longrightarrow \tL$ with $\gamma|_\L=\alpha$ and $\gamma|_{N_{\L^+}(Q)}=\alpha_Q$ for every $Q\in\Gamma_0$.
\end{lemma}

\begin{proof}
Write $\Pi^+\colon\D^+\longrightarrow \L^+$, $\Pi\colon\D\longrightarrow\L$ and $\tPi\colon\tD\longrightarrow\tL$ for the partial products on $\L^+$, $\L$ and $\tL$ respectively. Recall from Definition~\ref{D:RestrictionLocality} that then $\D\subseteq\D^+$ and $\Pi=\Pi^+|_\D$. As $\Gamma_0$ is a set of representatives of the $\F$-conjugacy classes of subgroups in $\Gamma^+\backslash\Gamma$, for every $P\in\Gamma^+\backslash \Gamma$, there is a unique element of $\Gamma_0\cap P^\F$, which we denote by $Q_P$. By Lemma~\ref{L:LocalityoverFFnatural}(b), for every $P\in\Gamma^+\backslash\Gamma$, we may moreover pick $h_P\in N_{\L^+}(N_S(P),S)$ with $P^{h_P}=Q_P$. As $S\in\Delta$, we have $T\in\Gamma$. So $P\neq T$ and thus $P<N_T(P)\in\Gamma$ by assumption. In particular, $N_S(P)\in\Delta$, $h_P\in N_\L(N_S(P),S)$ and the conjugate $P^{h_P}$ is defined in $\L$.

\smallskip

We define now first a map $\gamma\colon\L^+\longrightarrow \tL$ and show then that it has the required properties. If $f\in\L$, then we set $f\gamma=f\alpha$. Suppose now that $f\in\L^+\backslash\L$ so that $S_f\in\Delta^+\backslash\Delta$ and thus \[P:=S_f\cap T\in\Gamma^+\backslash\Gamma.\]
Notice that $P$ and $P^f$ are $\F$-conjugate and thus $Q:=Q_P=Q_{P^f}$. As 
\[u_f:=(h_P,h_P^{-1},f,h_{P^f},h_{P^f}^{-1})\in\D^+\mbox{ via }P,Q,P,P^f,Q,P^f,\]
we have  $f=\Pi^+(u_f)=\Pi^+(h_P,g,h_{P^f}^{-1})$, where $g:=\Pi^+(h_P^{-1},f,h_{P^f})\in N_{\L^+}(Q)$. 
Observe that $P^{h_P}=Q$ and $Q^{h_{P^f}^{-1}}=P^f$ in $\L$. Moreover,
\[g\alpha_Q\in N_{\L^+}(Q)\alpha_Q\subseteq N_{\tL}(Q\alpha).\]
Since $\Delta^+\alpha\subseteq\tDelta$ and $\alpha$ is a homomorphism of partial groups, we conclude that \[(h_P\alpha,g\alpha_Q,h_{P^f}^{-1}\alpha)\in \tD\mbox{ via }P\alpha,Q\alpha,Q\alpha,P^f\alpha.\]
So for every $f\in\L^+\backslash\L$, setting $P:=S_f\cap T$ and $Q:=Q_P$, we may define $f\gamma$ via
\begin{equation}\label{fgamma}
 f\gamma=\tPi(h_P\alpha,g\alpha_{Q},h_{P^f}^{-1}\alpha)\mbox{ where }g=\Pi^+(h_P^{-1},f,h_{P^f})\in N_{\L^+}(Q).
\end{equation}
If $\beta\colon\L^+\longrightarrow \tL$ is a homomorphism of partial groups with $\beta|_\L=\alpha$ and $\beta|_{N_{\L^+}(Q)}=\alpha_{Q}$ for all $Q\in\Gamma_0$, then with $f$, $g$, $P$ and $Q$ as above, we see that
\[f\beta=\Pi^+(h_P,g,h_{P^f}^{-1})\beta=\tPi(h_P\beta,g\beta,h_{P^f}^{-1}\beta)=\tPi(h_P\alpha,g\alpha_{Q},h_{P^f}^{-1}\alpha)=f\gamma.\]
Hence, we have in this case $\beta=\gamma$. So to prove the assertion, it is sufficient to show that $\gamma$ is a homomorphism of partial groups with $\gamma|_{N_{\L^+}(Q)}=\alpha_Q$ for all $Q\in\Gamma_0$.

\smallskip

\emph{Step~1:} Given $f\in\L^+$ and $P\in\Gamma^+\backslash\Gamma$ with $P\leq S_f\cap T$, we show that \eqref{fgamma} holds for $Q:=Q_P$. If $P=S_f\cap T$, then $S_f\not\in\Delta$, so $f\not\in\L$ and the equation holds by the definition of $\gamma$. 

\smallskip

Assume now $P<S_f\cap T$. Then by assumption $S_f\cap T\in\Gamma$ and thus $f\in\L$. Moreover, $P<R:=N_{S_f\cap T}(P)\in\Gamma$. Notice that $R\leq N_S(P)\leq S_{h_P}$ and $R^f\leq N_S(P^f)\leq S_{h_{P^f}}$. Hence, $(h_P,h_P^{-1},f,h_{P^f},h_{P^f}^{-1})\in\D$ via $R$, and so \[f=\Pi(h_P,h_P^{-1},f,h_{P^f},h_{P^f}^{-1})=\Pi(h_P,g,h_{P^f}^{-1}),\]
where $g:=\Pi(h_P^{-1},f,h_{P^f})=\Pi^+(h_P^{-1},f,h_{P^f})\in N_\L(Q)$. Recall that $\gamma|_\L=\alpha$ by definition of $\gamma$. As $\alpha$ is a homomorphism of partial groups with $\alpha|_{N_{\L}(Q)}=\alpha_Q|_{N_\L(Q)}$, it follows
\[f\gamma=f\alpha=\tPi(h_P\alpha,g\alpha,h_{P^f}^{-1}\alpha)=\tPi(h_P\alpha,g\alpha_Q,h_{P^f}^{-1}\alpha).\]
So \eqref{fgamma} holds.

\smallskip

\emph{Step~2:} We show that $\gamma|_{N_{\L^+}(Q)}=\alpha_Q$ for every $Q\in\Gamma_0$. To prove this fix $Q\in\Gamma_0$ and $f\in N_{\L^+}(Q)$. Observe that $h_Q\in N_{\L^+}(Q)$ and $h_Q\alpha_Q\in N_{\tL}(Q\alpha)$. Indeed, $h_Q\in N_\L(Q)$ and so $h_Q\alpha_Q=h_Q\alpha$. By Step~1, we have \[f\gamma=\tPi(h_Q\alpha,g\alpha_{Q},h_Q^{-1}\alpha)\] 
where $g=\Pi^+(h_Q^{-1},f,h_Q)\in N_{\L^+}(Q)$. Moreover, $u:=(h_Q\alpha,h_Q^{-1}\alpha,f\alpha_Q,h_Q\alpha,h_Q^{-1}\alpha)\in\W(N_{\tL}(Q\alpha))\subseteq\tD$. As $N_{\L^+}(Q)$ is a group, $f=\Pi^+(h_Q,h_Q^{-1},f,h_Q,h_Q^{-1})$. So
\[f\alpha_Q=\tPi(u)=\tPi(h_Q\alpha,\tPi(h_Q^{-1}\alpha,f\alpha_Q,h_Q\alpha),h_Q^{-1}\alpha)=\tPi(h_Q\alpha,g\alpha_Q,h_Q^{-1}\alpha)=f\gamma,\]
where the third equality uses that $\alpha_Q$ is a homomorphism of groups with $h_Q\alpha_Q=h_Q\alpha$ and $h_Q^{-1}\alpha_Q=h_Q^{-1}\alpha$. 

\smallskip

\emph{Step~3:} We show that $\gamma$ is a homomorphism of partial groups and thus the assertion holds by Step~2. For the proof let $w=(f_1,\dots,f_n)\in\D^+$. If $w\in\D$, then $\Pi^+(w)\gamma=\Pi(w)\alpha=\tPi(w\alpha^*)=\tPi(w\gamma^*)$ as $\gamma|_\L=\alpha$ is assumed to be a homomorphism of partial groups. Thus, we may assume $w\not\in\D$. Then $w\in\D^+$ via $P_0^*,P_1^*,\dots,P_n^*\in\Delta^+\backslash\Delta$. Thus, upon setting $P_i:=P_i^*\cap T$ for $i=0,1,\dots,n$, it follows from our assumption that 
\[w\in\D^+\mbox{ via }P_0,P_1,\dots,P_n\in\Gamma^+\backslash\Gamma.\]
Notice that $P_0,P_1,\dots,P_n$ are all $\F$-conjugate and so $Q:=Q_{P_0}=Q_{P_i}$ for $i=1,\dots,n$. Set $h_i:=h_{P_i}$ for $i=0,1,\dots,n$. By Step~1
\begin{equation}\label{E:fi}
f_i\gamma=\tPi(h_{i-1}\alpha,g_i\alpha_Q,h_i^{-1}\alpha)\mbox{ where }g_i=\Pi^+(h_{i-1}^{-1},f_i,h_i)\in N_{\L^+}(Q)
\end{equation}
for $i=1,\dots,n$. Set
\[u:=(h_0,g_1,h_1^{-1},h_1,g_2,h_2^{-1},\dots,h_{n-1},g_n,h_n^{-1})\mbox{ and }g:=\Pi^+(g_1,g_2,\dots,g_n).\]
Using Step~2, we see that $u\gamma^*=(h_0\alpha,g_1\alpha_Q,h_1^{-1}\alpha,h_1\alpha,g_2\alpha_Q,h_2^{-1}\alpha,\dots,h_{n-1}\alpha,g_n\alpha_Q,h_n^{-1}\alpha)$. Notice that $u\in\D^+$ via $P_0,Q,Q,P_1,\dots,P_{n-1},Q,Q,P_n$. Similarly, as $\alpha$ is a homomorphism of partial groups and $g_i\alpha_Q\in N_\tL(Q\alpha)$ for $i=1,\dots,n$, we have $u\gamma^*\in\tD$ via $P_0\alpha,Q\alpha,Q\alpha,P_1\alpha,\dots,P_{n-1}\alpha,Q\alpha,Q\alpha,P_n\alpha$. Using \eqref{E:fi} and applying axiom (PG3) of a partial group and Lemma~\ref{L:PartialGroup}(b) several times, we get that $w\gamma^*=(f_1\gamma,\dots,f_n\gamma)\in\tD$ and 
\begin{eqnarray*}
\tPi(w\gamma^*)&=&\tPi(u\gamma^*)\\
&=&\tPi(h_0\alpha,g_1\alpha_Q,\dots,g_n\alpha_Q,h_n^{-1}\alpha)\\
&=&\tPi(h_0\alpha,\tPi(g_1\alpha_Q,\dots,g_n\alpha_Q),h_n^{-1}\alpha)\\
&=&\tPi(h_0\alpha,g\alpha_Q,h_n^{-1}\alpha).
\end{eqnarray*}
Observe also that $f_i=\Pi^+(h_{i-1},h_{i-1}^{-1},f_i,h_i,h_i^{-1})=\Pi^+(h_{i-1},g_i,h_i^{-1})$ for $i=1,\dots,n$. So similarly, again using axiom (PG3) and Lemma~\ref{L:PartialGroup}(b) repeatedly, we see that 
\[f:=\Pi^+(w)=\Pi^+(u)=\Pi^+(h_0,g_1,\dots,g_n,h_n^{-1})=\Pi^+(h_0,g,h_n^{-1})\in N_{\L^+}(P_0,P_n).\]
As $g=\Pi^+(h_0^{-1},h_0,g,h_n^{-1},h_n)=\Pi^+(h_0^{-1},f,h_n)$, it follows from Step~1 that 
\[f\gamma=\tPi(h_0\alpha,g\alpha_Q,h_n^{-1}\alpha).\]
Putting everything together, we get $\Pi^+(w)\gamma=f\gamma=\tPi(w\gamma^*)$ and thus $\gamma$ is a homomorphism of partial groups. This completes Step~3 and the proof of the assertion.
\end{proof}

\begin{cor}\label{C:Main}
Let $(\L^+,\Delta^+,S)$ and $(\tL,\tDelta,\tS)$ be localities, and let $\Delta$ be a non-empty subset of $\Delta^+$ which is $\F_S(\L^+)$-closed. Set $\L:=\L^+|_\Delta$. Assume that every proper overgroup in $S$ of an element of $\Delta^+\backslash \Delta$ is in $\Delta$. Suppose we are given
\begin{itemize}
 \item a homomorphism of partial groups $\alpha\colon \L\longrightarrow \tL$ with $\Delta^+\alpha\subseteq\tDelta$; 
 \item a set $\Gamma_0\subseteq\Delta^+\backslash\Delta$ of fully $\F_S(\L^+)$-normalized representatives of the $\F_S(\L^+)$-conjugacy classes of the subgroups in $\Delta^+\backslash\Delta$; and
 \item for each $Q\in\Gamma_0$ a homomorphism of groups $\alpha_Q\colon N_{\L^+}(Q)\longrightarrow N_{\tL}(Q\alpha)$ with $\alpha_Q|_{N_{\L}(Q)}=\alpha|_{N_\L(Q)}$.
\end{itemize}
Then there exists a unique homomorphism of partial groups $\gamma\colon \L^+\longrightarrow \tL$ with $\gamma|_\L=\alpha$ and $\gamma|_{N_{\L^+}(Q)}=\alpha_Q$ for every $Q\in\Gamma_0$.
\end{cor}

\begin{proof}
 Apply Lemma~\ref{L:Main} with $S$ in place of $T$.
\end{proof}

In the proofs of our main theorems, Lemma~\ref{L:Main} will be used in the form of the following corollary.

\begin{cor}\label{C:MainCor}
Let $(\L^+,\Delta^+,S)$ and $(\L,\Delta,S)$ be linking localities over the same fusion system $\F$ such that  $\Delta\subseteq\Delta^+$ and $\L=\L^+|_\Delta$. Suppose we are given a locality $(\tL,\tDelta,\tS)$ and a homomorphism of partial groups $\alpha\colon \L\longrightarrow \tL$ with $\Delta^+\alpha\subseteq\tDelta$. Then there exists a unique homomorphism of partial groups $\gamma\colon \L^+\longrightarrow \tL$ with $\gamma|_\L=\alpha$.
\end{cor}

\begin{proof}
We first prove the assertion under the following additional hypothesis:
\begin{equation}\label{E:Star}\tag{$*$}
 \mbox{We have $Q\in\Delta$ for every $Q\leq S$ such that $Q$ properly contains a member of $\Delta^+\backslash\Delta$.}
\end{equation}
Notice that $\Delta^+\backslash\Delta$ is closed under $\F$-conjugacy, as $\Delta$ and $\Delta^+$ are closed under $\F$-conjugacy. Thus, we may choose a set $\Gamma_0\subseteq\Delta^+\backslash\Delta$ of fully $\F$-normalized representatives of the $\F$-conjugacy classes of the elements in $\Delta^+\backslash\Delta$. By Lemma~\ref{L:LinkingLocN(R)}, for every $Q\in\Gamma_0$, the normalizer $N_\L(Q)=N_{\L^+}(Q)$ is a subgroup of $\L$. As $\alpha$ is a homomorphism of partial groups, we have $N_\L(Q)\alpha\subseteq N_{\tL}(Q\alpha)$. Since we assume that $Q\alpha\in\Delta^+\alpha\subseteq\tDelta$, the normalizer $N_{\tL}(Q\alpha)$ is a subgroup of $\tL$. So $\alpha_Q:=\alpha|_{N_\L(Q)}$ is a homomorphism of groups from $N_\L(Q)=N_{\L^+}(Q)$ to $N_{\tL}(Q\alpha)$. Now the assertion follows from Corollary~\ref{C:Main}.  

\smallskip

We now remove the extra hypothesis \eqref{E:Star}. Let $\Gamma$ be the set of the elements in $\Delta^+\backslash\Delta$ of maximal order. Then $\Gamma$ is closed under $\F$-conjugacy, since $\Delta^+\backslash\Delta$ is closed under $\F$-conjugacy. Moreover, as $\Delta^+$ is overgroup closed in $S$, every proper overgroup of an element of $\Gamma$ is in $\Delta$. In particular, as $\Delta$ is $\F$-closed, the set $\Delta^*:=\Delta\cup\Gamma$ is $\F$-closed and $\L^*:=\L^+|_{\Delta}$ is well-defined. Then $(\L^*,\Delta^*,S)$ is a locality with $\F^{cr}\subseteq\Delta\subseteq\Delta^*$ and $\L^*|_\Delta=\L$. Moreover, $N_{\L^*}(P)=N_{\L^+}(P)$ is of characteristic $p$ for every $P\in\Delta^*$. Hence, using Alperin's fusion theorem \cite[Theorem~I.3.6]{AKO}, we conclude that $(\L^*,\Delta^*,S)$ is a linking locality over $\F$. Notice that \eqref{E:Star} is fulfilled with $\Delta^*$ in place of $\Delta^+$. As we proved at the beginning, there exists thus a unique homomorphism of partial groups $\gamma^*\colon\L^*\longrightarrow\tL$ with $\gamma^*|_\L=\alpha$. Now by induction on $|\Delta^+\backslash\Delta|$, there exists a unique homomorphism of partial groups $\gamma\colon\L^+\longrightarrow \tL$ with $\gamma|_{\L^*}=\gamma^*$. Then $\gamma|_\L=\gamma^*|_\L=\alpha$. Moreover, if $\beta\colon\L^+\longrightarrow\tL$ with $\beta|_\L=\alpha$, then $\beta^*:=\beta|_{\L^*}$ is a homomorphism of partial groups from $\L^*$ to $\tL$ with $\beta^*|_\L=\alpha$. As $\gamma^*$ is unique, it follows $\beta|_{\L^*}=\beta^*=\gamma^*$, and then the unique choice of $\gamma$ implies $\beta=\gamma$. This proves the assertion.   
\end{proof}

\section{Transporter systems}\label{S:Trans}

Transporter systems are certain categories associated to fusion systems which were introduced by Oliver and Ventura \cite{OV}. As shown by Chermak \cite[Appendix A]{Ch}, there is a one-to-one correspondence between localities and transporter systems, which we will outline in Subsection~\ref{SS:TransLoc}. We will moreover introduce isomorphisms between transporter systems in Subsection~\ref{SS:IsoTrans} and linking systems in Subsection~\ref{SS:LinkingSystems}. 

\smallskip

Since the literature on transporter systems is mainly written in ``left hand notation'', in this section we will write functions on the left side of the argument. Similarly, we will conjugate from the left. Given a group $G$, we set ${}^g\!x:=gxg^{-1}$ and ${}^g\!P=gPg^{-1}$ for all $x,g\in G$ and $P\subseteq G$. So conjugation by $g$ from the left corresponds to conjugation by $g^{-1}$ from the right.

\subsection{Isomorphisms between transporter systems}\label{SS:IsoTrans} 
If $\Delta$ is a set of subgroups of $G$, write $\T_\Delta(G)$ for the category whose object set is $\Delta$, and whose morphism set between two subgroups $P,Q\in\Delta$ is $N_G^l(P,Q):=\{g\in G\colon {}^g\!P\leq Q\}$ or, more precisely, the set of triples $(P,Q,g)$ with $g\in N_G^l(P,Q)$. Here composition of morphisms corresponds to the group multiplication.

\smallskip

A \emph{transporter system} associated to a fusion system $\F$ over $S$ is a category $\T$ whose set $\Delta$ of objects is an $\F$-closed collection of subgroups of $S$, together with functors
\[\T_\Delta(S)\xrightarrow{\;\;\;\delta\;\;\;}\T\xrightarrow{\;\;\;\pi\;\;\;}\F\]
subject to certain axioms. For example, $\delta$ is the identity on objects and injective on morphism sets, and $\pi$ is the inclusion on objects and surjective on morphism sets; see \cite[Definition~3.1]{OV} for details. If we want to be more precise, we say that $(\T,\delta,\pi)$ is a transporter system. By \cite[Lemma~A.6]{OV}, if $P\in\Delta$, then every element of $\Aut_\T(P):=\Mor_\T(P,P)$ is an isomorphism and so $\Aut_\T(P)$ is a group. We set $\delta_P:=\delta_{P,P}$ and $\pi_P:=\pi_{P,P}$ for every $P\in\Delta$. Similarly, we set $\alpha_P:=\alpha_{P,P}$ for every functor $\alpha$ from $\T$.

\smallskip

If $\T$ is a transporter system and $P,Q\in\ob(\T)$, then $\delta_{P,Q}(1)$ should be thought of as the inclusion map. Given $P,Q,P_0,Q_0\in\ob(\T)$ and $\psi\in\Mor_\T(P,Q)$ with $P_0\leq P$, $Q_0\leq Q$ and $\pi(\psi)(P_0)\leq Q_0$, by \cite[Lemma~3.2(c)]{OV}, there is a unique morphism $\psi_0\in\Mor_\T(P_0,Q_0)$ such that $\delta_{Q_0,Q}(1)\circ \psi_0=\psi\circ \delta_{P_0,P}(1)$. The morphism $\psi_0$ is then denoted by $\psi|_{P_0,Q_0}$ and called a restriction of $\psi$. On the other hand, if $\psi_0$ is given then, since every morphism in $\T$ is an epimorphism by \cite[Lemma~3.2(d)]{OV}, the ``extension'' $\psi$ is uniquely determined if it exists.

\begin{definition}\label{D:TransIso}
Let $(\T,\delta,\pi)$ and $(\tT,\tdelta,\tpi)$ be transporter systems associated to fusion systems $\F$ and $\tF$ over $p$-groups $S$ and $\tS$ respectively. 
\begin{itemize}
\item An equivalence of categories $\alpha\colon \T\longrightarrow\tT$ is called an \emph{isomorphism} if
\begin{itemize}
\item $\alpha$ is \emph{isotypical}, i.e. $\alpha_P(\delta_P(P))=\tdelta_{\alpha(P)}(\alpha(P))$ for every $P\in\ob(\T)$; and
\item $\alpha$ \emph{sends inclusions to inclusions}, i.e. $\alpha_{P,Q}(\delta_{P,Q}(1))=\tdelta_{\alpha(P),\alpha(Q)}(1)$ for all $P,Q\in\ob(\T)$. 
\end{itemize}
We write $\Iso(\T,\tT)$ for the set of isomorphisms from $\T$ to $\tT$. 
\item If $\Gamma$ and $\tGamma$ are sets of subgroups of $S$ and $\tS$ respectively, then we write $\Iso(\T,\tT)_{\Gamma,\tGamma}$ for the set of isomorphisms $\alpha\colon\T\longrightarrow\tT$ with $\{\alpha_S(\delta_S(P))\colon P\in\Gamma\}=\{\delta_\tS(Q)\colon Q\in\tGamma\}$.
\item An isomorphism $\alpha\colon\T\rightarrow\tT$ is called \emph{rigid} if $S=\tS$ and $\alpha_S\circ \delta_S=\tdelta_\tS$.
\item An \emph{automorphism} of $\T$ is an isomorphism $\T\longrightarrow\T$. Set $\Aut(\T):=\Iso(\T,\T)$ and, for any set $\Gamma$ of subgroups of $S$, set $\Aut(\T)_\Gamma:=\Iso(\T,\T)_{\Gamma,\Gamma}$.
\item If $\gamma\in\Aut_\T(S)$, then an automorphism $c_\gamma\in\Aut(\T)$ is defined on objects via $P\mapsto c_\gamma(P):=\pi(\gamma)(P)$ and on morphisms $\phi\in\Hom_\T(P,Q)$ by sending $\phi$ to
\[c_\gamma(\phi):=\gamma|_{Q,c_\gamma(Q)} \circ \phi \circ (\gamma|_{P,c_\gamma(P)})^{-1}\in\Hom_\T(c_\gamma(P),c_\gamma(Q)).\]
We will refer to $c_\gamma$ as the automorphism of $\T$ induced by \emph{conjugation by $\gamma$}. The group of automorphisms of $\T$ of the form $c_\gamma$ with $\gamma\in\Aut_\T(S)$ is called the group of \emph{inner automorphisms} of $\T$.
\item Let $\Out_\typ(\T)$ be the set of natural isomorphism classes of isotypical self-equivalences of $\L$.    
\end{itemize}
\end{definition}

It should be noted that $\Out_\typ(\T)$ is a submonoid of the (finite) group of natural isomorphism classes of self-equivalences of $\L$, and thus forms a finite group. In Lemma~\ref{L:ExactSequence} below, we will see that $\Out_\typ(\T)$ is actually the image of $\Aut(\T)$ under a homomorphism whose kernel is the group of inner automorphisms. 

\smallskip

The above definition of isomorphisms and rigid isomorphisms of transporter systems follows Glauberman--Lynd \cite[Definition~2.3]{GL2020}. The definition of isomorphisms used previously in the literature (e.g. in \cite[p.799]{BLO2}, \cite[p.146]{AKO} and \cite[Definition~A.2]{Ch}) is different. In the situations we care about (in particular when we consider linking systems later on), it agrees with the definition of a rigid isomorphism as we explain in Remark~\ref{R:Transporter}(d) below. The group $\Aut(\T)$ is also often denoted by $\Aut_\typ^I(\T)$; see e.g. \cite[p.153]{AKO}.

\begin{remark}\label{R:Transporter}
Suppose $(\T,\delta,S)$ and $(\tT,\tdelta,\tpi)$ are transporter systems over fusion systems $\F$ and $\tF$ respectively, and let $\alpha\colon\T\longrightarrow\tT$ be an isomorphism of transporter systems. 
\begin{itemize}
\item [(a)] It follows from the axioms of a transporter system that $\delta_S\colon S\longrightarrow \Aut_\T(S)$ is a group homomorphism whose image $\delta_S(S)$ is a normal Sylow $p$-subgroup of $\Aut_\T(S)$. Similarly, $\tdelta_\tS(\tS)$ is a normal Sylow $p$-subgroup of $\Aut_\tT(\tS)$. In particular
\[\alpha_S\colon \Aut_\T(S)\longrightarrow\Aut_\tT(\tS)\]
is an isomorphism of groups which takes $\delta_S(S)$ to $\tdelta_\tS(\tS)$. So writing $\tdelta_\tS^{-1}$ for the inverse of the map $\tS\rightarrow \tdelta(\tS)$ induced by $\tdelta_\tS$, the map $\beta:=\tdelta_\tS^{-1}\circ \alpha_S\circ \delta_S$ is a group isomorphism from $S$ to $\tS$. If $\Gamma$ and $\tGamma$ are sets of subgroups of $S$ and $\tS$ respectively, notice that $\alpha\in\Iso(\T,\tT)_\Gamma$ if and only if $\beta(\Gamma):=\{\beta(P)\colon P\in\Gamma\}$ equals $\tGamma$. 
\item [(b)] As can be seen from the proof of \cite[Proposition~2.5]{GL2020}, if $\F^{cr}\subseteq\ob(\T)$ and $\tF^{cr}\subseteq\ob(\tT)$, then the isomorphism $\beta\colon S\rightarrow\tS$ from (a) induces an isomorphism of fusion systems from $\F$ to $\tF$. In particular, if $\F^{cr}\subseteq\ob(\T)$ and $\Gamma$ is an $\Aut(\F)$-invariant set of subgroups of $S$, then $\Aut(\T)_\Gamma=\Aut(\T)$.  
\item [(c)] Since $\alpha$ maps inclusions to inclusions, $\alpha$ commutes with taking restrictions and extensions. If $P\in\ob(\T)$, then observe that $\delta_S(x)|_{P,P}=\delta_P(x)$ for every $x\in P$ and similarly, $\tdelta_\tS(y)|_{\alpha(P),\alpha(P)}=\tdelta_{\alpha(P)}(y)$ for every $y\in\alpha(P)$. Hence, as $\alpha_P(\delta_P(P))=\tdelta_{\alpha(P)}(\alpha(P))$, it follows $\alpha_S(\delta_S(P))=\tdelta_\tS(\alpha(P))$ for every $P\in\ob(\T)$. So if $\beta\colon S\rightarrow\tS$ is as in (a), one can see that $\alpha(P)=\beta(P)$ for every $P\in\ob(\T)$. Hence, $\alpha$ is a bijection on objects and thus an isomorphism of categories. In particular, $\Aut(\T)$ is a group. Moreover, if $\Gamma\subseteq\ob(\T)$ and $\tGamma\subseteq\ob(\tT)$, then $\alpha\in\Iso(\T,\tT)_{\Gamma,\tGamma}$ if and only if $\{\alpha(P)\colon P\in\Gamma\}=\tGamma$.
\item [(d)] If $\alpha$ is a rigid isomorphism, then $S=\tS$ and the isomorphism $\beta\colon S\rightarrow S$ from (a) is the identity. In particular, by (c), it follows then that $\alpha(P)=P$ for all $P\in\ob(\T)$. If $\F^{cr}\subseteq\T$ and $\tF^{cr}\subseteq\tT$, then it is a consequence of  \cite[Proposition~2.5]{GL2020} that $\alpha\circ \delta=\tdelta$ and $\tpi\circ \alpha=\pi$. Thus, in this case a rigid isomorphism $\T\longrightarrow\tT$ is the same as an isomorphism of transporter systems in the sense of \cite[Definition~A.2]{Ch} (which extends the  definition of an isomorphism of linking systems in \cite[p.~799]{BLO2} and \cite[p.~146]{AKO}).
\end{itemize} 
\end{remark}

\subsection{The correspondence between transporter systems and localities}\label{SS:TransLoc}

If $(\L,\Delta,S)$ is a locality, then one can easily define a transporter system $\T_\Delta(\L)$ over $\F_S(\L)$ similarly as for groups; the object set of $\T_\Delta(\L)$ is $\Delta$, and the morphism set between two objects $P$ and $Q$ consists of the triples $(P,Q,g)$, where $g\in\L$ with $P\subseteq\D(g^{-1})$ and ${}^g\!P:=P^{g^{-1}}\leq Q$. We will outline now how one can construct a locality from a transporter system.

\smallskip

Let $(\T,\delta,\pi)$ be a transporter system associated to a fusion system $\F$. By $\Delta$ denote the set of objects of $\T$, write $\Iso_\T(P,Q)$ for the set of isomorphisms between two objects $P$ and $Q$, and $\Iso(\T)$ for the set of all isomorphisms in $\T$. By \cite{Ch}, there is a partial order $\uparrow_\T$ defined on $\Iso(\T)$ via $\phi_0\uparrow_\T\phi$ if $\phi_0\in\Iso_\T(P_0,Q_0)$, $\phi\in\Iso_\T(P,Q)$, $P_0\leq P$, $Q_0\leq Q$ and 
\[\phi\circ\delta_{P_0,P}(1)=\delta_{Q_0,Q}(1)\circ\phi_0.\]
Note that the latter condition means that $\phi_0=\phi|_{P_0,Q_0}$.

\begin{definition}
Let $\L_\Delta(\T)$ be the set of equivalence classes of the elements of $\Iso(\T)$ with respect to the smallest equivalence relation on $\Iso(\T)$ containing $\uparrow_\T$. If $\phi\in\Iso(\T)$, write $[\phi]$ for the equivalence class of $\phi$ in $\L_\Delta(\T)$. By $\D$ denote the set of tuples $w=(f_k,f_{k-1},\dots,f_1)\in\W(\L_\Delta(\T))$ for which there exist $\phi_i\in f_i$ for $i=1,\dots,k$ such that the composition $\phi_k\circ\phi_{k-1}\circ\cdots\circ \phi_1$ is defined in the category $\T$. Moreover, given such $w$ and $\phi_i$, set $\Pi(w):=[\phi_k\circ\phi_{k-1}\circ\cdots\circ \phi_1]$.
\end{definition}

The map $\Pi\colon \D\longrightarrow\L_\Delta(\T)$ defined above is well-defined. Together with $\Pi$ and the map $\L\longrightarrow\L,[\phi]\mapsto[\phi^{-1}]$ (which is also well-defined), the set $\L$ forms a partial group by \cite[Proposition~A.9]{Ch}.  Moreover, the map
\[S\longrightarrow\L_\Delta(\T),\;x\mapsto [\delta_S(x)]\]
is an injective homomorphism of partial groups, and its image $[S]$ is a subgroup of $\L_\Delta(\T)$. Most of the time, we will identify $x\in S$ with $[\delta_S(x)]\in [S]$. With this identification, by \cite[Proposition~A.13]{Ch}, $(\L_\Delta(\T),\Delta,S)$ is a locality.

\begin{lemma}\label{L:BasicLDeltaT}
Let $(\T,\delta,\pi)$ be a transporter system associated to a fusion system $\F$, $\Delta:=\ob(\T)$ and $\L:=\L_\Delta(\T)$. As above write $[\phi]$ for the equivalence class of $\phi\in\Iso(\T)$ in $\L$. Then the following hold:
\begin{itemize}
 \item [(a)] If $P,Q\in\Delta$, $\phi\in\Iso_\T(P,Q)$ and $f=[\phi]$, then $P\subseteq S_{f^{-1}}$, ${}^f\!P=Q$ and $c_{f^{-1}}|_P=\pi(\phi)$.
 \item [(b)] $\Aut_\T(P)\cong N_\L(P)$.
 \item [(c)] $\F_S(\L)$ is the subsystem of $\F$ generated by all the sets $\Hom_\F(P,Q)$ with $P,Q\in\Delta$. 
\end{itemize}
\end{lemma}

\begin{proof}
Let $P$, $Q$, $\phi$ and $f$ be as in (a), fix $x\in P$ and set $y:=\pi(\phi)(x)\in Q$. Observe that, via the usual identification of the elements of $S$ with elements of $\L$, we have $x=[\delta_S(x)]=[\delta_P(x)]$, since $\delta_S(x)\delta_{P,S}(1)=\delta_{P,S}(x)=\delta_{P,S}(1)\delta_P(x)$. Similarly  $y=[\delta_S(y)]=[\delta_Q(y)]$. Notice that the composition $\phi\circ \delta_P(x)\circ \phi^{-1}$ is defined in $\L$. Moreover, it follows from the definition of a transporter system (axiom (C) in \cite[Definition~3.1]{OV}) that $\phi\circ \delta_P(x)\circ \phi^{-1}=\delta_Q(y)$. Hence, $(f,x,f^{-1})=([\phi],[\delta_P(x)],[\phi^{-1}])\in\D$ and $x^{f^{-1}}=\Pi(f,x,f^{-1})=[\phi\circ \delta_P(x)\circ \phi^{-1}]=[\delta_Q(y)]=y$. This shows (a). 

\smallskip

Property (a) yields in particular that the map $\alpha\colon \Aut_\T(P)\longrightarrow N_\L(P),\phi\mapsto[\phi]$ is well-defined. Moreover, $\alpha$ is  surjective by \cite[Corollary~A.11]{Ch} and injective by \cite[Lemma~A.8(b)]{Ch}. For all $\phi,\psi\in\Aut(\T)$, we have $\alpha(\phi\circ\psi)=[\phi\circ\psi]=\Pi([\phi],[\psi])=\Pi(\alpha(\phi),\alpha(\psi))$. Hence, $\alpha$ is an isomorphism of groups and (b) holds. 

\smallskip

To prove (c) notice that $\F_S(\L)$ is generated by all the maps of the form $c_{f^{-1}}\colon P\longrightarrow Q$, where $P,Q\in\Delta$, $P\leq S_{f^{-1}}$ and ${}^f\!P=Q$. For such $P,Q,f$, by \cite[Corollary~A.11]{Ch}, there exists always $\phi\in\Iso_\T(P,Q)$ with $f=[\phi]$. 
Moreover, the fusion system generated by the sets $\Hom_\F(P,Q)$ with $P,Q\in\Delta$ is actually generated by the sets $\Iso_\F(P,Q)$ with $P,Q\in\Delta$. Since $\pi$ is surjective on morphism sets and, by \cite[Lemma~A.6]{OV}, the preimages of isomorphisms in $\F$ under $\pi$ are isomorphisms in $\T$, property (c) follows from (a).
\end{proof}

If $\C$ is a small category and $\Gamma\subseteq\ob(\C)$, we will write $\C|_\Gamma$ for the full subcategory of $\C$ with object set $\Gamma$.

\begin{lemma}\label{L:Iota}
Let $\T^+$ be a transporter system associated to a fusion system $\F$ and let $\Delta\subseteq\Delta^+:=\ob(\T^+)$ such that $\Delta$ is $\F$-closed. Then  $\T:=\T^+|_\Delta$ is a transporter system associated to $\F$. Moreover, writing $[\phi]_+$ for the equivalence class of $\phi\in\Iso(\T^+)$ in $\L_{\Delta^+}(\T^+)$, and $[\phi]$ for the equivalence class of $\phi\in\Iso(\T)$ in $\L_\Delta(\T)$, the map
\[\iota\colon \L_\Delta(\T)\longrightarrow \L_{\Delta^+}(\T^+)|_\Delta,\;[\phi]\mapsto [\phi]_+\mbox{ for all }\phi\in\Iso(\T)\]
is well-defined and an isomorphism of partial groups, which restricts to the identity on $S$ (if one identifies the elements of $S$ with elements of $\L_\Delta(\T)$ and $\L_{\Delta^+}(\T^+)$ as usual). 
\end{lemma}

\begin{proof}
As $\Delta$ is $\F$-closed, it is immediate from the axioms of a transporter system that $\T:=\T^+|_\Delta$ together with the restriction of $\delta$ to $\T_\Delta(S)$ and the restriction of $\pi$ to $\T$ is a transporter system. Set $\L:=\L_\Delta(\T)$ and $\L^+:=\L_{\Delta^+}(\T^+)$. Write $\D$ for the domain of the partial product on $\L$ and $\D':=\D_\Delta(\L^+)$ for the domain of the partial product on $\L^+|_\Delta$.

\smallskip

If $\phi,\psi\in\Iso(\T)\subseteq\Iso(\T^+)$, then $\phi\uparrow_\T\psi$ implies $\phi\uparrow_{\T^+}\psi$ and so $[\phi]=[\psi]$ yields $[\phi]_+=[\psi]_+$. Hence, the map
\[\iota'\colon \L\longrightarrow \L^+,\;[\phi]\mapsto [\phi]_+\mbox{ for all }\phi\in\Iso(\T)\]
is well-defined.  It follows from the construction of the partial products on $\L=\L_\Delta(\T)$ and $\L^+=\L_{\Delta^+}(\T^+)$ that $\iota$ is a homomorphism of partial groups. Moreover, since we identify every element $x\in S$ with $[\delta_S(x)]\in\L$ and with $[\delta_S(x)]_+\in\L^+$, the map $\iota'$ restricts to the identity on $S$. In particular, as $(\L,\Delta,S)$ is a locality, it follows $\iota'(\L)\subseteq\L^+|_\Delta$ and $(\iota')^*(\D)\subseteq\D'=\D_\Delta(\L^+)$ (with $(\iota')^*$ defined as in Definition~\ref{N:PartialHomWordMap}, but written on the left). Hence, $\iota$ is well-defined and a homomorphism of partial groups. It remains to show that $\iota$ is injective and $\D'\subseteq\iota^*(\D)$. 

\smallskip

If $\phi\uparrow_{\T^+}\chi$ for some $\phi\in\Iso(\T)$ and $\chi\in\Iso(\T^+)$, then the assumption that $\Delta$ is overgroup closed in $S$ implies $\chi\in\Iso(\T)$ and $\phi\uparrow_\T\chi$. By \cite[Lemma~A.8(a)]{Ch}, every element of $\L^+$ contains a unique maximal element with respect to the partial order $\uparrow_{\T^+}$. So if $[\phi]_+=[\psi]_+$ for some $\phi,\psi\in\Iso(\T)$, then for the $\uparrow_{\T^+}$-maximal element $\chi$ of $[\phi]_+=[\psi]_+$, we have $\chi\in\Iso(\T)$, $\phi\uparrow_\T\chi$ and $\psi\uparrow_\T\chi$. Hence, $[\phi]=[\psi]$ which proves that $\iota$ is injective.

\smallskip
 
If $P,Q\in\Delta$ and $f\in\L^+$ with ${}^f\!P=Q$, then by \cite[Corollary~A.11]{Ch}, there exists $\psi\in\Iso_{\T^+}(P,Q)$ with $f=[\psi]_+$. For such $\psi$, we have $\psi\in\Iso_\T(P,Q)$ and $\iota([\psi])=f$. From this property, the definition of $\D$ and the definition of $\D'=\D_\Delta(\L^+)$, one sees that $\D'\subseteq\iota^*(\D)$. Hence the assertion holds.
\end{proof}

\begin{lemma}\label{L:Lambda}
Let $(\T,\delta,\pi)$ and $(\tT,\tdelta,\tpi)$ be transporter systems associated to fusion systems $\F$ and $\tF$ over $p$-groups $S$ and $\tS$ respectively. Set $\Delta:=\ob(\T)$, $\L:=\L_\Delta(\T)$, $\tDelta=\ob(\tT)$ and $\tL:=\L_{\tDelta}(\tT)$. For $\alpha\in\Iso(\T,\tT)$ define $\Lambda(\alpha)\colon \L\longrightarrow \tL$ to be the map which, for all $P,Q\in\Delta$ and all $\phi\in\Iso_{\T}(P,Q)$, sends the class $[\phi]\in\L$ to the class $[\alpha_{P,Q}(\phi)]\in\tL$. Then this defines a bijection
\[\Lambda\colon \Iso(\T,\tT)\longrightarrow \Iso((\L,\Delta,S),(\tL,\tDelta,\tS)),\;\alpha\mapsto\Lambda(\alpha).\]
Moreover, if $\Gamma$ and $\tGamma$ are sets of subgroups of $S$ and $\tS$ respectively, then $\Lambda$ induces a bijection
\[\Iso(\T,\tT)_{\Gamma,\tGamma}\longrightarrow \Iso((\L,\Delta,S),(\tL,\tDelta,\tS))_{\Gamma,\tGamma}.\]
\end{lemma}

\begin{proof}
By \cite[Theorem~2.11]{GL2020} and the proof of this result, there is an equivalence $\Lambda'$ from the category of transporter systems with isomorphisms to the category of localities with isomorphisms, which is defined on objects by sending a transporter system $\T$ to $\L_{\ob(\T)}(\T)$, and on morphisms by sending an isomorphism $\alpha\in\Iso(\T,\tT)$ to $\Lambda(\alpha)$ as defined in the theorem. In particular, $\Lambda=\Lambda'_{\T,\tT}$ is a bijection $\Iso(\T,\tT)\longrightarrow\Iso((\L,\Delta,S),(\tL,\tDelta,\tS))$.

\smallskip

Let now $P\leq S$ and $Q\leq \tS$. Via the usual identifications of the elements of $S$ and $\tS$ with elements of $\L$ and $\tL$, we have 
\[P=\{[\delta_S(x)]\colon x\in P\}\mbox{ and }Q=\{[\tdelta_\tS(y)]\colon y\in Q\}.\]
So given $\alpha\in\Iso(\T,\tT)$,
\[\Lambda(\alpha)(P)=\{\Lambda(\alpha)([\delta_S(x)])\colon x\in P\}=\{[\alpha_S(\delta_S(x))]\colon x\in P\}.\]
As the map $\Aut_{\tT}(S)\longrightarrow\tL,\phi\mapsto[\phi]$ is by \cite[Lemma~A.8(b)]{Ch} injective, it follows that $\Lambda(\alpha)(P)=Q$ if and only if $\alpha_S(\delta_S(P))=\tdelta_\tS(Q)$. So $\Lambda(\alpha)(\Gamma):=\{\Lambda(\alpha)(P)\colon P\in\Gamma\}$ equals $\tGamma$ if and only if $\{\alpha_S(\delta_S(P))\colon P\in\Gamma\}=\{\tdelta_\tS(Q)\colon Q\in\tGamma\}$. Equivalently, $\Lambda(\alpha)\in\Iso((\L,\Delta,S),(\tL,\tDelta,\tS))_{\Gamma,\tGamma}$ if and only if $\alpha\in\Iso(\T,\tT)_{\Gamma,\tGamma}$. (Unlike in Definition~\ref{D:LocalityHomomorphism}, we write maps here on the right.)
\end{proof}

\subsection{Linking systems}\label{SS:LinkingSystems}

In this paper we work with the following definition of a linking system, which is slightly non-standard, but fits well with the earlier given definition of a linking locality (Definition~\ref{D:LinkingLoc}).

\begin{definition}\label{D:LinkingSystem}
If $\F$ is a saturated fusion system, then a \emph{linking system} associated to $\F$ is a transporter system $\T$ associated to $\F$ such that $\F^{cr}\subseteq\ob(\T)$ and $\Aut_\T(P)$ is of characteristic $p$ for every $P\in\ob(\T)$. If $\ob(\T)=\F^s$, then $\T$ is called a \emph{subcentric linking system} associated to $\F$.
\end{definition}

The original definitions of linking systems in \cite{BLO2}, \cite{controlling} and \cite{OliverExtensions} are not based on the definition of a transporter systems. A linking system in either of these definitions is a linking system in the above definition, while the converse does not hold in general. Historically, \emph{centric linking systems}, i.e. linking systems over $\F$ whose object sets are the sets of $\F$-centric subgroups, were studied first. The longstanding conjecture that there is a centric linking system associated to every saturated fusion system, and that such a centric linking system is unique up to a rigid isomorphism was shown by Chermak \cite{Ch}, and subsequently by Oliver \cite{Oliver:2013}. Originally, these proofs use the classification of finite simple groups, but the dependence on the classification of the proof in \cite{Oliver:2013} was removed by Glauberman--Lynd \cite{GL2016}.

\smallskip

If $\T$ is a linking system associated to a saturated fusion system $\F$ over $S$, then $\ob(\T)\subseteq\F^s$. On the other hand, if $\F^{cr}\subseteq\Delta\subseteq\F^s$ such that $\Delta$ is $\F$-closed, then it is stated in \cite[Theorem~A]{subcentric} that there is a linking system $\T$ with object set $\Delta$ associated to $\F$; moreover, such $\T$ is unique up to rigid isomorphism. The proof relies heavily on the existence and uniqueness of centric linking systems. Formally, \cite[Theorem~A]{subcentric} is proved as a consequence of the corresponding statement about linking localities which is summarized in Subsection~\ref{SS:LinkingLoc}. We use this opportunity to point out that a precise argument that $\T$ is unique up to a rigid isomorphism is actually missing in \cite{subcentric}. However, the uniqueness of $\T$ follows from \cite[Theorem~2.11]{GL2020} (or \cite[Lemma~A.14, Lemma~A.15]{Ch}) and from Lemma~\ref{L:LinkingSystemLinkingLoc} below.) 

\smallskip

If $(\L,\Delta,S)$ is a locality over $\F$, then it is easy to see that the corresponding transporter system $\T_\Delta(\L)$ is a linking system associated to $\F$ if and only if $(\L,\Delta,S)$ is a linking locality. Moreover, we have the following lemma.

\begin{lemma}\label{L:LinkingSystemLinkingLoc}
If $\T$ is a linking system associated to a saturated fusion system $\F$, then for $\Delta:=\ob(\T)$, the locality $(\L_\Delta(\T),\Delta,S)$ is a linking locality over $\F$.  
\end{lemma}

\begin{proof}
As $\F^{cr}\subseteq\Delta$, it follows from Alperin's fusion theorem \cite[Theorem~I.3.6]{AKO} and Lemma~\ref{L:BasicLDeltaT}(c) that $\F_S(\L)=\F$. In particular, $\F_S(\L)$ is saturated and $\F_S(\L)^{cr}\subseteq\Delta$. Moreover, by Lemma~\ref{L:BasicLDeltaT}(b), $N_\L(P)\cong\Aut_\T(P)$ is of characteristic $p$ for every $P\in\Delta$.
\end{proof}

\begin{lemma}\label{L:LinkingSystemsElementary}
If $(\T,\delta,\pi)$ is a linking system associated to a saturated fusion system $\F$ over $S$, then the following hold:
\begin{itemize}
 \item [(a)] $\ker(\pi_S)=\delta_S(Z(S))$.
 \item [(b)] For every $P\in\ob(\T)$, $O_p(\Aut_\T(P))=\delta_P(P)$ if and only if $P\in\F^{cr}$.
 \item [(c)] \emph{(Alperin's Fusion Theorem for linking systems)} Each morphism in $\T$ is the composite of restrictions of elements in the  automorphism groups $\Aut_\T(P)$, where $P\in\F^{cr}$ is fully $\F$-normalized. 
\end{itemize}
\end{lemma}

\begin{proof}
If $\gamma\in\Aut_\T(S)$, then for all $g\in S$, $\pi_S(\gamma)(g)=g$ if and only if $\gamma$ commutes with $\delta_S(g)$ by Axiom~C in \cite[Definition~3.1]{OV} and \cite[Lemma~3.3]{OV}. Hence, $\ker(\pi_S)=C_{\Aut_\T(S)}(\delta_S(S))$. As $\Aut_\T(S)$ is of characteristic $p$ and $\delta_S(S)$ is a normal Sylow $p$-subgroup of $\Aut_\T(S)$, we have $\gamma\in C_{\Aut_\T(S)}(\delta_S(S))=Z(\delta_S(S))=\delta_S(Z(S))$ showing (a).

\smallskip

Property (b) follows from \cite[Lemma~6.2]{subcentric} and Lemma~\ref{L:LinkingSystemLinkingLoc}, or alternatively this property can be shown by reformulating the argument in the proof of \cite[Lemma~6.2]{subcentric} for transporter systems. Property (c) follows from (b) and \cite[Proposition~3.9]{OV}.
\end{proof}

\begin{lemma}\label{L:ExactSequence}
If $(\T,\delta,\pi)$ is a linking system associated to a saturated fusion system $\F$ over $S$, then the  sequence
\[1\longrightarrow Z(\F)\xrightarrow{\;\;\delta_S\;\;} \Aut_\T(S)\xrightarrow{\;\gamma\mapsto c_\gamma\;} \Aut(\T)\longrightarrow \Out_{\typ}(\T)\longrightarrow 1  
\]
is exact.
\end{lemma}

\begin{proof}
The statement was shown in \cite[Lemma~1.14(a)]{AOV1} for linking systems in Oliver's definition, i.e. for linking systems whose objects are quasicentric subgroups. The argument can be repeated verbatim (with $\L$ replaced by $\T$) to prove exactness in $\Aut(\T)$ and in $\Out_\typ(\T)$ and to show that $c_\gamma=\id_\T$ implies $\gamma\in\delta_S(Z(\F))$; here only the reference to \cite[Lemma~1.11(b')]{AOV1} needs to be replaced by a reference to \cite[Lemma~3.2(c)]{OV}, the reference to axiom (A) needs to be replaced by a reference to Lemma~\ref{L:LinkingSystemsElementary}(a), and the reference to \cite[Lemma~1.11(e)]{AOV1} needs to be replaced by a reference to axiom (II) in the definition of a transporter system \cite[Definition~3.1]{OV}. On the other hand, by \cite[Proposition~4.5]{AKO}, we have $Z(\F)\leq P$ for all $P\in\F^{cr}$. So if $a\in Z(\F)$, then Lemma~\ref{L:LinkingSystemsElementary}(c) yields that any morphism $\psi\in\Mor_\L(P,Q)$ extends to a morphism $\ov{\psi}\in\Mor_\L(\<P,a\>,\<Q,a\>)$ with $\pi(\ov{\psi})(a)=a$. Such $\ov{\psi}$ commutes with $\gamma=\delta_S(a)$ by axiom (C) again. So $\gamma$ commutes with $\psi$ and thus $c_\gamma=\id_\L$. This shows exactness in $\Aut_\T(S)$.
\end{proof}

\section{Isomorphisms between linking localities and linking systems}\label{S:Iso}

In this section we prove Theorems~\ref{T:A1} and \ref{T:A2}. Moreover, we show considerably more general versions of these theorems, where for each  result we formulate a version for linking localities and a version for linking systems. Theorem~\ref{T:A2} leads naturally to a statement about outer automorphism groups (Theorem~\ref{T:A2Outer}), and building on this we prove Theorem~\ref{C:B}.

\begin{theorem}[Linking locality version]\label{T:MainIso}
Let $(\L,\Delta,S)$, $(\L^+,\Delta^+,S)$, $(\tL,\tDelta,\tS)$ and $(\tL^+,\tDelta^+,\tS)$ be linking localities such that 
\begin{itemize}
\item $\F_S(\L^+)=\F_S(\L)$, $\Delta\subseteq\Delta^+$, $\L=\L^+|_\Delta$, and
\item $\F_{\tS}(\tL^+)=\F_{\tS}(\tL)$, $\tDelta\subseteq\tDelta^+$, $\tL=\tL^+|_{\tDelta}$. 
\end{itemize}
Then the map
\[\Psi\colon \Iso((\L^+,\Delta^+,S),(\tL^+,\tDelta^+,\tS))_{\Delta,\tDelta}\longrightarrow \Iso((\L,\Delta,S),(\tL,\tDelta,\tS))_{\Delta^+,\tDelta^+}\]
with $\Psi(\gamma)=\gamma|_{\L}$ is well-defined and a bijection. 
\end{theorem}

\begin{proof}
By Lemma~\ref{L:IsoRestrictionLocality}, the map $\Psi$ is well-defined. If $\alpha\in \Iso((\L,\Delta,S),(\tL,\tDelta,\tS))_{\Delta^+,\tDelta^+}$, then $\alpha$ regarded as a map $\L\longrightarrow\tL^+$ is a homomorphism of partial groups. Thus, by Corollary~\ref{C:MainCor}, $\alpha$ extends to a unique homomorphism of partial groups $\gamma\colon \L^+\longrightarrow\tL^+$. By Lemma~\ref{L:PartialIso}, $\alpha^{-1}$ is a homomorphism of partial groups from $\tL$ to $\L$, which can be regarded as a homomorphism of partial groups $\tL\longrightarrow\L^+$. So again by Corollary~\ref{C:MainCor}, $\alpha^{-1}$ extends to a homomorphism of partial groups $\hat{\gamma}\colon\tL^+\longrightarrow\L^+$. Then $\gamma\hat{\gamma}\colon\L^+\longrightarrow\L^+$ and $\hat{\gamma}\gamma\colon\tL^+\longrightarrow\tL^+$ are  homomorphisms of partial groups with $(\gamma\hat{\gamma})|_{\L}=\alpha\alpha^{-1}=\id_{\L}$ and $(\hat{\gamma}\gamma)|_{\tL}=\alpha^{-1}\alpha=\id_{\tL}$. It follows from   Corollary~\ref{C:MainCor} applied with $\id_{\L}$ in place of $\alpha$ that there is a unique homomorphism of partial groups $\L^+\longrightarrow\L^+$ which restricts to the identity on $\L$. Thus, any such homomorphism equals $\id_{\L^+}$. Similarly, any homomorphism of partial groups $\tL^+\longrightarrow\tL^+$ which restricts to the identity on $\tL$ equals $\id_{\tL^+}$. This shows  $\gamma\hat{\gamma}=\id_{\L^+}$ and $\hat{\gamma}\gamma=\id_{\tL^+}$, i.e. $\gamma$ is bijective with inverse map $\hat{\gamma}$. So $\gamma\colon\L^+\longrightarrow\tL^+$ is an isomorphism of partial groups by Lemma~\ref{L:PartialIso}. As $\Delta^+\gamma=\Delta^+\alpha=\tDelta^+$ and $\Delta\gamma=\Delta\alpha=\tDelta$, it follows that $\gamma\in \Iso((\L^+,\Delta^+,S),(\tL^+,\tDelta^+,\tS))_{\Delta,\tDelta}$ with $\Psi(\gamma)=\gamma|_{\L}=\alpha$. This shows that $\Psi$ is surjective. As $\gamma$ is the unique homomorphism of partial groups $\L^+\longrightarrow\tL^+$ which restricts to $\alpha$, the maps $\Psi$ is also injective.  
\end{proof}

\begin{theorem}[Linking system version]\label{T:MainIsoT}
Suppose $\T$, $\T^+$, $\tT$ and $\tT^+$ are linking systems, and $\F$ and $\tF$ are saturated fusion systems such that 
\begin{itemize}
\item $\T$ and $\T^+$ are linking systems associated to $\F$, $\ob(\T)\subseteq\ob(\T^+)$, $\T=\T^+|_{\ob(\T)}$;
\item $\tT$ and $\tT^+$ are linking systems associated to $\tF$, $\ob(\tT)\subseteq\ob(\tT^+)$, $\tT=\tT^+|_{\ob(\tT)}$.
\end{itemize}
Then the map
\[\Iso(\T^+,\tT^+)_{\ob(\T),\ob(\tT)}\longrightarrow \Iso(\T,\tT)_{\ob(\T^+),\ob(\tT^+)},\;\alpha\mapsto\alpha|_\T\]
is a bijection.
\end{theorem}

\begin{proof}
As we are dealing with transporter systems, in this proof, we will again write functions from the left. Set $\Delta:=\ob(\T)$, $\Delta^+:=\ob(\T^+)$, $\tDelta:=\ob(\tT)$ and $\tDelta^+:=\ob(\tT^+)$, 
\[\L:=\L_\Delta(\T),\;\L^+:=\L_{\Delta^+}(\T^+),\;\tL:=\L_{\tDelta}(\tT)\mbox{ and }\tL^+:=\L_{\tDelta^+}(\tT^+).\]
By $[\phi]$ we denote the  equivalence class of $\phi$ in $\L$ if $\phi\in\Iso(\T)$, and the equivalence class of $\phi$ in $\tL$ if $\phi\in\Iso(\tT)$. Similarly, $[\phi]_+$ denotes the equivalence class of $\phi$ in $\L^+$ if $\phi\in\Iso(\T^+)$ and the equivalence class of $\phi$ in $\tL^+$ if $\phi\in\Iso(\tT^+)$. By Lemma~\ref{L:Iota}, the maps $\iota\colon \L\longrightarrow \L^+|_\Delta,[\phi]\mapsto [\phi]_+$ and $\tiota\colon\tL\longrightarrow\tL^+|_{\tDelta},[\phi]\mapsto[\phi]_+$ are isomorphisms of localities which restrict to the identity on $S$. In particular, the map 
\[\Phi\colon \Iso((\L,\Delta,S),(\tL,\tDelta,\tS))_{\Delta^+,\tDelta^+}\longrightarrow \Iso((\L^+|_\Delta,\Delta,S),(\tL^+|_{\tDelta},\tDelta,\tS))_{\Delta^+,\tDelta^+},\;\beta\mapsto \tiota\circ\beta\circ \iota^{-1}\]
is a bijection. By Lemma~\ref{L:IsoRestrictionLocality}, there is also a bijection
\[\Psi\colon \Iso((\L^+,\Delta^+,S),(\tL^+,\tDelta^+,\tS))_{\Delta,\tDelta}\longrightarrow \Iso((\L^+|_\Delta,\Delta,S),(\tL^+|_{\tDelta},\tDelta,\tS))_{\Delta^+,\tDelta^+}\]
given by restriction. By Lemma~\ref{L:Lambda}, there are moreover bijections
\[\Lambda\colon \Iso(\T,\tT)_{\Delta^+,\tDelta^+}\longrightarrow \Iso((\L,\Delta,S),(\tL,\tDelta,\tS))_{\Delta^+,\tDelta^+}\]
and
\[\Lambda^+\colon \Iso(\T^+,\tT^+)_{\Delta,\tDelta}\longrightarrow \Iso((\L^+,\Delta^+,S),(\tL^+,\tDelta^+,\tS))_{\Delta,\tDelta}.\]
Here $\Lambda$ is defined by 
\[\Lambda(\alpha)([\phi])=[\alpha_{P,Q}(\phi)]\]
for all $\alpha\in\Iso(\T,\tT)_{\Delta^+,\tDelta^+}$, all $P,Q\in\Delta$ and all  $\phi\in\Iso_\T(P,Q)$, and $\Lambda^+$ is defined by
\[\Lambda^+(\alpha)([\phi]_+)=[\alpha_{P,Q}(\phi)]_+\]
for all $\alpha\in\Iso(\T^+,\tT^+)_{\Delta,\tDelta}$, all $P,Q\in\Delta^+$ and all $\phi\in\Iso_{\T^+}(P,Q)$.  

\smallskip

Now $\Psi\circ\Lambda^+$ is a bijection from $\Iso(\T^+,\tT^+)_{\Delta,\tDelta}$ to $\Iso((\L^+|_\Delta,\Delta,S),(\tL^+|_{\tDelta},\tDelta,\tS))_{\Delta^+,\tDelta^+}$ and $\Phi\circ\Lambda$ is a bijection from $\Iso(\T,\tT)_{\Delta^+,\tDelta^+}$ to $\Iso((\L^+|_\Delta,\Delta,S),(\tL^+|_{\tDelta},\tDelta,\tS))_{\Delta^+,\tDelta^+}$. Hence, $\Theta:=(\Phi\circ\Lambda)^{-1}\circ (\Psi\circ \Lambda^+)$ is a bijection from $\Iso(\T^+,\tT^+)_{\Delta,\tDelta}$ to $\Iso(\T,\tT)_{\Delta^+,\tDelta^+}$. Fixing $\alpha\in\Iso(\T^+,\tT^+)_{\Delta,\tDelta}$, it only remains to show that $\Theta(\alpha)=\alpha|_\T$, or equivalently, $\Psi(\Lambda^+(\alpha))=\Phi(\Lambda(\alpha|_\T))$. To prove the latter equality, recall that $\iota\colon \L\longrightarrow\L^+|_\Delta$ is bijective. So every element of $\L^+|_\Delta$ is of the form $\iota([\phi])=[\phi]_+$ for some $P,Q\in\Delta$ and $\phi\in\Iso_\T(P,Q)$. We compute then
\[\Psi(\Lambda^+(\alpha))[\phi]_+=\Lambda^+(\alpha)[\phi]_+=[\alpha_{P,Q}(\phi)]_+\]
and
\begin{eqnarray*}
\Phi(\Lambda(\alpha|_\T))[\phi]_+ &=& (\tiota\circ\Lambda(\alpha|_\T)\circ \iota^{-1})[\phi]_+\\
&=& (\tiota\circ\Lambda(\alpha|_\T))[\phi]\\
&=&\tiota([\alpha_{P,Q}(\phi)])=[\alpha_{P,Q}(\phi)]_+. 
\end{eqnarray*}
This proves the assertion.
\end{proof}

The two preceding theorems seem most important in situations where we consider automorphisms of linking localities or linking systems. In the next two theorems we state the results for automorphisms explicitly.

\begin{theorem}[Linking locality version]\label{T:MainAut}
Let $(\L,\Delta,S)$ and $(\L^+,\Delta^+,S)$ be linking localities over the same fusion system $\F$ such that $\Delta\subseteq\Delta^+$ and $\L=\L^+|_\Delta$. Then the map
\[\Psi\colon\Aut(\L^+,\Delta^+,S)_\Delta\longrightarrow \Aut(\L,\Delta,S)_{\Delta^+},\;\gamma\mapsto\gamma|_\L\]
is well-defined and an isomorphism of groups. In particular, if $\Delta$ and $\Delta^+$ are $\Aut(\F)$-invariant, then the map $\Aut(\L^+,\Delta^+,S)\longrightarrow\Aut(\L,\Delta,S),\;\gamma\mapsto\gamma|_\L$ is an isomorphism of groups.
\end{theorem}

\begin{proof}
 By Theorem~\ref{T:MainIso}, the map $\Psi$ is well-defined and a bijection. Moreover, if $\beta,\gamma\in\Aut(\L^+,\Delta^+,S)_\Delta$, then $\Psi(\beta\gamma)=(\beta\gamma)|_\L=(\beta|_\L)(\gamma|_\L)=\Psi(\beta)\Psi(\gamma)$. Hence, $\Psi$ is an isomorphism of groups. By Lemma~\ref{L:ProjectionLocalityMorphismFusionSystem}, if $\alpha$ is an element of $\Aut(\L,\Delta,S)$ or of $\Aut(\L^+,\Delta^+,S)$, then $\alpha|_S\in\Aut(\F)$. Hence, if $\Delta$ and $\Delta^+$ are $\Aut(\F)$-invariant, then $\Aut(\L^+,\Delta^+,S)_\Delta=\Aut(\L^+,\Delta^+,S)$ and $\Aut(\L,\Delta,S)_{\Delta^+}=\Aut(\L,\Delta,S)$. This yields the assertion.
\end{proof}

\begin{theorem}[Linking system version]\label{T:MainAutTrans}
Let $\T$ and $\T^+$ be linking systems associated to the same fusion system $\F$ such that $\ob(\T)\subseteq\ob(\T^+)$ and $\T=\T^+|_{\ob(\T)}$. Then the map
\[\Theta\colon \Aut(\T^+)_{\ob(\T)}\longrightarrow\Aut(\T)_{\ob(\T^+)},\;\alpha\mapsto \alpha|_\T\]
is an isomorphism of groups. In particular, if $\ob(\T)$ and $\ob(\T^+)$ are $\Aut(\F)$-invariant, then the map $\Aut(\T^+)\longrightarrow\Aut(\T),\;\alpha\mapsto\alpha|_\T$ is an isomorphism of groups.  
\end{theorem}

\begin{proof}
By Theorem~\ref{T:MainIsoT}, $\Theta$ is a bijection, and it is easy to to see that $\Theta$ is an isomorphism of groups. As explained in Remark~\ref{R:Transporter}(b), if $\ob(\T)$ and $\ob(\T^+)$ are $\Aut(\F)$-invariant, then $\Aut(\T)_{\ob(\T^+)}=\Aut(\T)$ and $\Aut(\T^+)_{\ob(\T)}=\Aut(\T^+)$.
\end{proof}

\begin{proof}[Proof of Theorem~\ref{T:A1}]
Let $(\L,\Delta,S)$ and $(\L^+,\Delta^+,S)$ be linking localities over the same fusion system $\F$ such that $\Delta$ and $\Delta^+$ are $\Aut(\F)$-invariant. By \cite[Theorem~A(b)]{subcentric}, there exists a subcentric linking locality $(\L^s,\F^s,S)$ over $\F$. As $\F^{cr}\subseteq\Delta$ and $\F^{cr}\subseteq\Delta^+$, it follows that $(\L^s|_\Delta,\Delta,S)$ and $(\L^s|_{\Delta^+},\Delta^+,S)$ are linking localities over $\F$. Hence, by \cite[Theorem~A(a)]{subcentric}, there exist rigid isomorphisms from $(\L^s|_\Delta,\Delta,S)$ to $(\L,\Delta,S)$ and from $(\L^s|_{\Delta^+},\Delta^+,S)$ to $(\L^+,\Delta^+,S)$. By \cite[Lemma~3.6]{subcentric}, $\F^s$ is $\Aut(\F)$-invariant. Hence, applying Theorem~\ref{T:MainAut} twice, we obtain
\[\Aut(\L,\Delta,S)\cong\Aut(\L^s|_\Delta,\Delta,S)\cong\Aut(\L^s,\F^s,S)\cong\Aut(\L^s|_{\Delta^+},\Delta^+,S)\cong \Aut(\L^+,\Delta^+,S).\]
So Theorem~\ref{T:A1} follows from Theorem~\ref{T:MainAut}.
\end{proof}

We will prove Theorem~\ref{T:A2} together with the following similar statement about outer automorphism groups, which is a generalization of \cite[Lemma~1.17]{AOV1}. 

\begin{theorem}\label{T:A2Outer}
If $\T$ and $\T^+$ are linking system associated to $\F$ such that $\ob(\T)$ and $\ob(\T^+)$ are $\Aut(\F)$-invariant, then \[\Out_\typ(\T^+)\cong \Out_\typ(\T).\]
If $\ob(\T)\subseteq\ob(\T^+)$ and $\T=\T^+|_{\ob(\T)}$, then an isomorphism $\Out_\typ(\T^+)\xrightarrow{\;\;\cong\;\;}\Out_\typ(\T)$ is given by sending the class of $\alpha\in\Aut(\T^+)$ to the class of $\alpha|_\T\in\Aut(\T)$.  
\end{theorem}

\begin{proof}[Proof of Theorem~\ref{T:A2} and Theorem~\ref{T:A2Outer}]
Let $\T$ and $\T^+$ be transporter systems over the same fusion system $\F$ such that $\ob(\T)$ and $\ob(\T^+)$ are $\Aut(\F)$-invariant. As usual when dealing with transporter systems, we write maps on the left side of the argument. 

\smallskip

If $\ob(\T)\subseteq\ob(\T^+)$ and $\T=\T^+|_{\ob(\T)}$, then by Theorem~\ref{T:MainAutTrans}, the map $\Aut(\T^+)\rightarrow\Aut(\T),\gamma\mapsto\gamma|_\T$ is a group isomorphism. As $\Aut_\T(S)=\Aut_{\T^+}(S)$, one easily observes that it induces an isomorphism between the group of inner automorphisms of $\T^+$ and the group of inner automorphisms of $\T$. Hence, by Lemma~\ref{L:ExactSequence}, it induces an isomorphism $\Out_\typ(\T^+)\rightarrow \Out_\typ(\T)$ which takes the class of $\alpha\in\Aut(\T^+)$ to the class of $\alpha|_\T$. 

\smallskip

Suppose now that $\T$ and $\T^+$ are arbitrary. By \cite[Theorem~A]{subcentric}, there exists a subcentric linking system $\T^s$ over $\F$; moreover, $\T$ is rigidly isomorphic to $\T^s|_{\ob(\T)}$, and $\T^+$ is rigidly isomorphic to $\T^s|_{\ob(\T^+)}$. If $\alpha\colon\T\rightarrow\T^s|_{\ob(\T)}$ is a rigid isomorphism, then the map $\Phi\colon\Aut(\T)\rightarrow\Aut(\T^s|_{\ob(\T)}),\beta\mapsto \alpha\circ\beta\circ\alpha^{-1}$ is an isomorphism of groups and so $\Aut(\T)\cong \Aut(\T^s|_{\ob(\T)})$. One can check now  that, for any $\gamma\in\Aut(\T)$, we have $\alpha\circ c_\gamma\circ\alpha^{-1}=c_{\alpha_S(\gamma)}$; to see that $\alpha\circ c_\gamma\circ\alpha^{-1}$ and $c_{\alpha_S(\gamma)}$ agree on objects, one uses that $\tpi\circ \alpha=\pi$ (cf. Remark~\ref{R:Transporter}(d)), and to see that the two functors agree on morphisms, one uses that $\alpha$ takes inclusions to inclusions and thus commutes with taking restrictions. So $\Phi$ induces an isomorphism between the group of inner automorphisms of $\T$ and the group of inner automorphisms of $\T^s|_{\ob(\T)}$. Thus, by Lemma~\ref{L:ExactSequence},  $\Out_\typ(\T)\cong \Out_\typ(\T^s|_{\ob(\T)})$. Similarly, one shows that $\Aut(\T^+)\cong \Aut(\T^s|_{\ob(\T^+)})$ and $\Out_\typ(\T^+)\cong\Out_\typ(\T^s|_{\ob(\T^+)})$. So using Theorem~\ref{T:MainAutTrans} twice, we can conclude that 
\[\Aut(\T)\cong \Aut(\T^s|_{\ob(\T)})\cong \Aut(\T^s)\cong\Aut(\T^s|_{\ob(\T^+)})\cong \Aut(\T^+).\] 
and similarly
\[\Out_\typ(\T)\cong \Out_\typ(\T^s|_{\ob(\T)})\cong \Out_\typ(\T^s)\cong\Out_\typ(\T^s|_{\ob(\T^+)})\cong \Out_\typ(\T^+).\] 
\end{proof}

\begin{remark}\label{R:AOV}
Theorem~\ref{T:A2} and Theorem~\ref{T:A2Outer} were shown for linking systems whose objects are quasicentric in \cite[Lemma~1.17]{AOV1} and its proof via more direct arguments. As we will briefly indicate now, the proof could be adapted to give a proof of Theorems~\ref{T:A2} and \ref{T:A2Outer}, which does not use linking localities:
\begin{itemize}
 \item Using the notation in the proof of \cite[Lemma~1.17]{AOV1}, Lemma~\ref{L:LinkingSystemsElementary}(b) is needed to conclude that $P$ is properly contained in $\hat{P}$. The reference to \cite[Theorem~1.12]{AOV1} needs to be replaced by a reference to Lemma~\ref{L:LinkingSystemsElementary}(c). 
 \item The references to Proposition~1.11(b),(b') and Proposition~1.11(d) in \cite{AOV1} need to be replaced by references to Lemma~3.2(c) and Proposition~3.4(a) in \cite{OV} respectively.
 \item References to  \cite[Proposition~1.11(e)]{AOV1} could be replaced by references to Axiom II in the definition of a transporter system \cite[Definition~3.1]{OV} and to \cite[Lemma~3.3]{OV}.
 \item The reference to \cite[Lemma~1.15]{AOV1} can be replaced by a reference to \cite[Proposition~2.5]{GL2020} (cf. Remark~\ref{R:Transporter}(d)).
\end{itemize}
It seems that the arguments could also be adapted to give direct proofs of the more general Theorems~\ref{T:MainIsoT} and \ref{T:MainAutTrans}.
\end{remark}

\begin{proof}[Proof of Theorem~\ref{C:B}]
If $\T$ and $\T^+$ are linking systems associated to the same saturated fusion system $\F$ such that $\ob(\T)\subseteq\ob(\T^+)$ and $\T=\T^+|_{\ob(\T)}$, then by \cite[Theorem~A]{subcentric}, the inclusion map $\iota\colon \T\hookrightarrow\T^+$ induces a homotopy equivalence $|\iota|\colon |\T|\rightarrow |\T^+|$ and thus a homotopy equivalence $|\iota|^\wedge_p\colon |\T|^\wedge_p\rightarrow |\T^+|^\wedge_p$. Moreover, if $\ob(\T)$ and $\ob(\T^+)$ are $\Aut(\F)$-invariant and $\gamma\in\Aut(\T^+)$, then the commutative square
\[
\xymatrix{
\T^+ \ar[r]^{\gamma} & \T^+\\
\T\ar@{^{(}->}[u]^{\iota} \ar[r]^{\gamma|_\T} & \T \ar@{^{(}->}[u]_{\iota}
}
\]
induces a commutative square after applying the functor $|\cdot|^\wedge_p$. Thus, if $\ob(\T)$ and $\ob(\T^+)$ are $\Aut(\F)$-invariant, then by Theorem~\ref{T:A2Outer}, the conclusion of Corollary~\ref{C:B} is true for $\T$ if and only if it is true with $\T^+$ in place of $\T$. 

\smallskip

Suppose now $\T$ is an arbitrary linking system associated to $\F$ such that $\ob(\T)$ is $\Aut(\F)$-invariant. By \cite[Theorem~A]{subcentric}, there exists a subcentric linking system $\T^s$ associated to $\F$ such that $\T^s|_{\ob(\T)}=\T$; moreover, $\T^c:=\T^s|_{\F^c}$ is a centric linking system associated to $\F$. By \cite[Theorem~8.1]{BLO2} and its proof, the statement in Theorem~\ref{C:B} is true for $\T^c$ in place of $\T$. The object sets $\ob(\T^c)=\F^c$ and $\ob(\T^s)=\F^s$ are $\Aut(\F)$-invariant (cf. \cite[Lemma~3.6]{subcentric}). Hence, as remarked above, the conclusion of Theorem~\ref{C:B} is true for $\T^s$ and thus also for $\T$.    
\end{proof}

\section{Partial normal subgroups}\label{S:PartialNormal}

This section is mainly devoted to the proof of Theorem~\ref{T:MainChermakII}. We will however start in the first subsection with some background on partial normal subgroups of localities. Most importantly, we prove with Lemma~\ref{L:PartialNormalAlperin} a result which seems to be of general interest and can be considered as a version of Alperin's Fusion Theorem for partial normal subgroups. This lemma is also applied in \cite{normal}.  Using Lemma~\ref{L:PartialNormalAlperin} we will then prove Theorem~\ref{T:MainChermakII} and a corollary in Subsection~\ref{SS:PartialNormalLinking}.

\subsection{General results}\label{SS:PartialNormal}
If $\alpha\colon \L\longrightarrow\tL$ is a homomorphism of partial groups, then by \cite[Lemma~1.14]{loc1}, $\ker(\alpha)$ is a partial normal subgroup of $\L$. The other way around, if $(\L,\Delta,S)$ is a locality and $\N$ is a partial normal subgroup of $\L$, then one can construct a partial group $\L/\N$ and a projection of partial groups
\[\alpha\colon\L\longrightarrow \L/\N\]
with $\ker(\alpha)=\N$. We refer the reader to Lemma~3.16 and the preceding explanations in \cite[Section~3]{loc1} for details of the construction. We will often adopt a ``bar notation'' similarly as for groups. This means that, setting $\ov{\L}=\L/\N$, for every element or subset $X$ of $\L$, we write $\ov{X}$ for the image of $X$ in $\ov{\L}$ under $\alpha$. Moreover, for any set $\Gamma$ of subgroups of $\L$, we set $\ov{\Gamma}:=\{\ov{P}\colon P\in\Gamma\}$. By \cite[Corollary~4.5]{loc1}, $(\ov{\L},\ov{\Delta},\ov{S})$ is a locality and $\alpha$ is a projection of localities from $(\L,\Delta,S)$ to $(\ov{\L},\ov{\Delta},\ov{S})$.

\begin{lemma}\label{ProductPrepare}
Let $(\L,\Delta,S)$ be a locality with a partial normal subgroup $\N$. Then the following hold:
\begin{itemize}
 \item [(a)] The triple $(\N S,\Delta,S)$ is a locality.
 \item [(b)] For every $P\in\Delta$, $O^p(N_{\N S}(P))=O^p(N_\N(P))$.
\end{itemize}
\end{lemma}

\begin{proof}
Part (a) is true by \cite[Lemma~4.1]{loc1}. In particular, $N_{\N S}(P)$ is a subgroup of $\L$. Moreover, we may consider the canonical projection $\alpha\colon \N S\longrightarrow \ov{\N S}:=\N S/\N$. Then $\ov{\N S}=\ov{S}$ is a $p$-group and $\alpha$ induces a group homomorphism $\alpha|_{N_{\N S}(P)}\colon N_{\N S}(P)\longrightarrow \ov{S}$. Thus, $O^p(N_{\N S}(P))\leq \ker(\alpha|_{N_{\N S}(P)})=\ker(\alpha)\cap N_{\N S}(P)=N_\N(P)$. This implies (b).
\end{proof}

For the next lemma recall that a subgroup $H$ of a finite group $G$ is called strongly $p$-embedded if $H\neq G$, $p$ divides $|H|$ and  $H\cap H^g$ is a $p^\prime$-group for all $g\in G\backslash H$. %If $G$ has a strongly $p$-embedded subgroup, then $O_p(G)=1$ (see e.g. \cite[Proposition~A.7(c)]{AKO}). 

\begin{lemma}\label{L:PartialNormalAlperin}
Let $(\L,\Delta,S)$ be a locality. If $\N$ is a partial normal subgroup of $\L$ and $n\in\N$, then there exist $k\in\mathbb{N}$, $R_1,R_2,\dots,R_k\in \Delta$ and $(t,n_1,n_2,\dots,n_k)\in\D$ such that the following hold:
\begin{itemize}
\item[(i)] $S_n=S_{(t,n_1,\dots,n_k)}$ and $n=tn_1n_2\cdots n_k$;
\item[(ii)] $n_i\in O^p(N_\N(R_i))$, $S_{n_i}=R_i$, $O_p(N_{\N S}(R_i))=R_i$ and $N_S(R_i)\in\Syl_p(N_{\N S}(R_i))$ for all $i=1,\dots,k$; and
\item[(iii)] $t\in \N\cap S$.
\end{itemize}
In fact, $R_1,\dots,R_k$ can be chosen such that $N_{\N S}(R_i)/R_i$ has a strongly $p$-embedded subgroup for all $i=1,\dots,k$. 
\end{lemma}

\begin{proof}
By Lemma~\ref{ProductPrepare}(a), $(\N S,\Delta,S)$ is a locality. So by Alperin's fusion theorem for localities \cite[Theorem~2.5]{Molinier:2016}, there exist $k\in\mathbb{N}$, $Q_1,Q_2,\dots,Q_k\in \Delta$ and $(g_1,g_2,\dots,g_k)\in\D$ such that the following hold:
\begin{itemize}
\item $S_n=S_{(g_1,\dots,g_k)}$ and $n=g_1g_2\cdots g_k$;
\item $g_i\in N_{\N S}(Q_i)$, $S_{g_i}=Q_i$, $N_{\N S}(Q_i)/Q_i$ has a strongly $p$-embedded subgroup and $N_S(Q_i)\in\Syl_p(N_{\N S}(Q_i))$ for all $i=1,\dots,k$. 
\end{itemize}
As $N_S(Q_i)\in\Syl_p(N_{\N S}(Q_i))$, it follows from Lemma~\ref{ProductPrepare}(b) that $N_{\N S}(Q_i)=N_S(Q_i)O^p(N_\N(Q_i))$. So for all $i=1,\dots,n$, we can write $g_i=s_im_i$ with $s_i\in N_S(Q_i)$ and $m_i\in O^p(N_\N(Q_i))$. By Lemma~\ref{L:NLSbiset}, we have $Q_i=S_{g_i}=S_{(s_i,m_i)}$ for all $i=1,\dots,k$. In particular, $S_n=S_{(g_1,g_2,\dots,g_k)}=S_{(s_1,m_1,s_2,m_2,\dots,s_k,m_k)}$ and 
\[w:=(s_1,m_1,s_2,m_2,\dots,s_k,m_k)\in\D\mbox{ via }S_n.\]
Note also $n=\Pi(g_1,\dots,g_n)=\Pi(w)$. Set
\[x_i:=s_{i}s_{i+1}\dots s_k\mbox{ for all }1\leq i\leq k,\] 
\[n_i:=m_i^{x_{i+1}}\mbox{ for all }1\leq i<k\mbox{ and }n_k:=m_k.\] 
Notice that
\[x_ix_{i+1}^{-1}=s_i\mbox{ for }1\leq i<k\mbox{ and }x_k=s_k.\]
Since $w\in\D$ via $S_n$, it follows that
\[v:=(x_1,x_2^{-1},m_1,x_2,x_3^{-1},m_2,x_3,\dots,x_k^{-1},m_{k-1},x_k,m_k)\in\D\mbox{ via }S_n\]
and so, setting $t:=x_1$, also 
\[u:=(t,n_1,n_2,\dots,n_{k-1},n_k)\in\D\mbox{ via }S_n.\]
Observe moreover that, by axiom (PG3), we have
\[n=\Pi(w)=\Pi(v)=\Pi(u).\]
Using Lemma~\ref{L:LocalitiesProp}(f), it follows $S_n=S_u$. So (i) holds. Set 
\[R_i:=Q_i^{x_{i+1}}\mbox{ for }1\leq i<k\mbox{ and }R_k:=Q_k.\]
Recall that $m_i\in O^p(N_\N(Q_i))$, $N_{\N S}(Q_i)/Q_i$ has a strongly $p$-embedded subgroup and $N_S(Q_i)\in\Syl_p(N_{\N S}(Q_i))$. Therefore, Lemma~\ref{L:LocalitiesProp}(b) gives that $n_i\in O^p(N_\N(R_i))$, $N_{\N S}(R_i)/R_i$ has a strongly $p$-embedded subgroup and $N_S(R_i)\in\Syl_p(N_{\N S}(R_i))$ for $i=1,\dots,k$. In particular, $O_p(N_{\N S}(R_i)/R_i)=1$ and thus $O_p(N_{\N S}(R_i))=R_i$ (cf. \cite[Proposition~A.7(c)]{AKO}). 

\smallskip

For all $i=1,\dots,k$ it follows from Lemma~\ref{L:NLSbiset} that $Q_i=S_{g_i}=S_{m_i}^{s_i^{-1}}$. As $s_i\in N_S(Q_i)$, this implies
\[Q_i=S_{m_i}\mbox{ for all }i=1,\dots,k.\]
In particular, $R_k=Q_k=S_{m_k}=S_{n_k}$. Moreover, since $(m_i,Q_i)$ is conjugate to $(n_i,R_i)$ under $x_{i+1}\in S$ for $1\leq i<k$, Lemma~\ref{L:NLSbiset} gives also $R_i=S_{n_i}$ for $1\leq i<k$. So (ii) holds. Note now that $u':=(t^{-1},t,n_1,\dots,n_k,n^{-1})\in\D$ via $S_n^t$ and $t^{-1}=\Pi(t^{-1},n,n^{-1})=\Pi(u')=\Pi(n_1,\dots,n_k,n^{-1})\in S\cap\N$. Hence, $t\in \N\cap S$ and the proof is complete.
\end{proof}

\subsection{Partial normal subgroups of linking localities}\label{SS:PartialNormalLinking}

In this subsection, we will first prove Theorem~\ref{T:MainChermakII}. Afterwards, we prove as a corollary that any two linking localities over the same fusion system have the same number of partial normal subgroups. Against our usual convention, we will use the left hand notation for the map $\Phi_{\L^+,\L}$ from Theorem~\ref{T:MainChermakII}. Recall that $\fN(\L)$ denotes the set of partial normal subgroups of a partial groups $\L$. We first show the following lemma.

\begin{lemma}\label{L:ChermakIIHelp}
Let $(\L,\Delta,S)$ and $(\L^+,\Delta^+,S)$ be linking localities over the same fusion system $\F$ such that $\Delta\subseteq\Delta^+$ and $\L=\L^+|_\Delta$. Assume that every proper overgroup of an element of $\Delta^+\backslash \Delta$ is in $\Delta$. Let $\N^+\in\fN(\L^+)$, $\N:=\N^+\cap \L\in\fN(\L)$ and set $T:=\N^+\cap S=\N\cap S$. Then the following hold. 
\begin{itemize}
 \item [(a)] $\N^+=\<\N^{\L^+}\>$, where $\<\N^{\L^+}\>$ denotes the smallest partial subgroup of $\L^+$ containing all the elements of the form $n^f$ with $f\in\L^+$ and $n\in\N\cap\D^+(f)$.
 \item [(b)] If $\F_{S\cap\N}(\N)$ is $\F$-invariant, then $\F_T(\N)=\F_T(\N^+)$. 
 \item [(c)] Let $\K^+\in\fN(\L^+)$ and $\K:=\K^+\cap\L\in\fN(\L)$. Then $\K^+T=\N^+$ if and only if $\K T=\N$. 
\end{itemize}
\end{lemma}

\begin{proof}
Observe that, for any $k\in \L^+$ and $t\in S$, we have $S_{kt}=S_{(k,t)}=S_k$. Hence, $kt\in\L$ if and only if $k\in\L$. With $\K^+$ and $\K$ as in (c), it follows $\K^+T\cap\L=(\K^+\cap\L)T=\K T$. Hence, if $\K^+ T=\N^+$, then $\N=\N^+\cap\L=(\K^+ T)\cap\L=\K T$. On the other hand, if $\K T=\N$ and (a) holds, then $\K\subseteq\N$ and so $\K^+=\<\K^{\L^+}\>\subseteq\<\N^{\L^+}\>=\N^+$. Thus, since $\N^+$ is a partial subgroup, it follows in this case $\K^+T\subseteq\N^+$. Hence, it remains to prove (a), (b) and the following property:
\begin{equation}\label{E:Provec}
\mbox{If $\K^+$ and $\K$ are as in (c) and $\K T=\N$, then }\N^+\subseteq \K^+ T.
\end{equation}
Set $\E:=\F_T(\N)$. As $\N\subseteq\N^+\unlhd\L^+$, we have $\<\N^{\L^+}\>\subseteq\N^+$. Moreover, clearly  $\E=\F_T(\N)\subseteq\F_T(\N^+)$. So fixing $n\in\N^+$, we need to show that $n\in\<\N^{\L^+}\>$ and, if $\E$ is $\F$-invariant, then $c_n\colon S_n\cap T\longrightarrow T$ is a morphism in $\E$. Furthermore, fixing $\K^+$ and $\K$ are as in (c) such that $\K T=\N$, we need to show $n\in \K^+ T$. 

\smallskip

As $S_{(k,t)}=S_{kt}$ for all $k\in\K^+$ and $t\in T$, using the Frattini calculus \cite[Lemma~3.4]{loc1}, one sees that $\K^+ T=T\K^+$ is a subgroup of $\L^+$. So by Lemma~\ref{L:PartialNormalAlperin} applied with $\L^+$ and $\N^+$ in place of $\L$ and $\N$, we may assume that $P=S_n\in\Delta^+$ and $n\in O^p(N_{\N^+}(P))$. If $P\in\Delta$, then $N_{\L^+}(P)=N_\L(P)$ and $n\in N_{\N^+}(P)=N_\N(P)$ since $\N^+\cap\L=\N$. So in this case, $n\in\<\N^{\L^+}\>$, the conjugation homomorphism $c_n|_{P\cap T}$ is a morphism in $\E$, and $n\in \N=\K T\subseteq \K^+ T$. 

\smallskip

Suppose now that $P\in \Delta^+\backslash\Delta$. Then by Lemma~\ref{L:LocalityoverFFnatural}(a), there exists $f\in\L^+$ such that $P\leq S_f$ and $R:=P^f$ is fully $\F$-normalized.  By Lemma~\ref{L:LocalitiesProp}(b), the conjugation map $c_f\colon N_{\L^+}(P)\longrightarrow N_{\L^+}(R)$ is defined and an isomorphism of groups. In particular, $n^f\in O^p(N_{\N^+}(P))^f\subseteq O^p(N_{\N^+}(R))=O^p(N_\N(R))$, where the last equality uses $\L\cap\N^+=\N$ and $N_{\L^+}(R)=N_\L(R)$ by Lemma~\ref{L:LinkingLocN(R)}. Hence, using Lemma~\ref{L:LocalitiesProp}(c), we see that $n=(n^f)^{f^{-1}}\subseteq N_\N(R)^{f^{-1}}\subseteq\<\N^{\L^+}\>$ proving (a). 

\smallskip

As $\N=\K T\subseteq\K S$, Lemma~\ref{ProductPrepare}(b) applied with $\K$ in place of $\N$ gives that $n^f\in O^p(N_\N(R))\subseteq O^p(N_{\K S}(R))=O^p(N_\K(R))$. Hence, we have $n=(n^f)^{f^{-1}}\in \<\K^{\L^+}\>\subseteq\K^+\subseteq\K^+ T$ proving \eqref{E:Provec} and thus (c).  

\smallskip

For the proof of (b) note that $c_{n^f}|_{R\cap T}\in\Aut_\E(R\cap T)$. Define $\phi:=c_f|_{P\cap T}\in\Hom_\F(P\cap T,R\cap T)$. For every $x\in R$, we have $(f^{-1},n,f,x,f^{-1},n,f)\in\D^+$ via $R$, and so $x^{n^f}=((x^{f^{-1}})^n)^f$. Hence, $\phi^{-1}(c_n|_{P\cap T})\phi=c_{n^f}|_{R\cap T}\in\Aut_\E(R\cap T)$. If $\E$ is $\F$-invariant, using the characterization of $\F$-invariant subsystems given in \cite[Proposition~I.6.4(d)]{AKO}, we can conclude that $c_n|_{P\cap T}=\phi (c_{n^f}|_{R\cap T}) \phi^{-1}\in\Aut_\E(P\cap T)$. This shows (b) and completes the proof. 
\end{proof}

\begin{proof}[Proof of Theorem~\ref{T:MainChermakII}]
For every partial normal subgroup $\N^+$ of $\L^+$, it is easy to see that the intersection $\N^+\cap\L$ is a partial normal subgroup of $\L$. Hence, the map 
\[\Phi_{\L^+,\L}\colon\fN(\L^+)\longrightarrow \fN(\L),\;\N^+\mapsto \N^+\cap\L\]
is well-defined. Moreover, this map is clearly inclusion preserving. 

\smallskip

Without loss of generality, assume that $\Delta\neq\Delta^+$. Let $R\in\Delta^+\backslash \Delta$ be of maximal order. As $\Delta^+$ and $\Delta$ are closed under $\F$-conjugacy, $\Delta^*:=\Delta\cup R^\F$ is closed under $\F$-conjugacy and contained in $\Delta^+$. If $P$ is a proper overgroup of an element of $R^\F$, then $P\in\Delta^+$ as $\Delta^+$ is overgroup closed, so the maximality of $|R|$ yields $P\in\Delta$. Since $\Delta$ is overgroup closed, this shows that $\Delta^*$ is $\F$-closed and $\L^*:=\L^+|_{\Delta^*}$ is well-defined. Notice that $\F^{cr}\subseteq\Delta\subseteq \Delta^*$ and $N_{\L^*}(P)=N_{\L^+}(P)$ is of characteristic $p$ for every $P\in\Delta^*$. Therefore, $(\L^*,\Delta^*,S)$ is a linking locality over $\F$. So similarly, we have maps
\[\Phi_{\L^+,\L^*}\colon \fN(\L^+)\longrightarrow\fN(\L^*),\;\N^+\mapsto \N^+\cap\L^*\]
and
\[\Phi_{\L^*,\L}\colon \fN(\L^*)\longrightarrow\fN(\L),\;\N^*\mapsto \N^*\cap\L\]
defined. Notice that $\Phi_{\L^+,\L}=\Phi_{\L^*,\L}\circ \Phi_{\L^+,\L^*}$. By induction on $|\Delta^+\backslash\Delta|$, we may assume that the assertion is true with $(\L^*,\Delta^*,S)$ in place of $(\L,\Delta,S)$. That means that $\Phi_{\L^+,\L^*}$ is a bijection  such that $\Phi_{\L^+,\L^*}^{-1}$ is inclusion preserving; moreover, given $\N^+\unlhd \L^+$ and $\N^*=\N^+\cap\L^*\unlhd\L^*$ such that $\F_{S\cap\N^*}(\N^*)$ is normal in $\F$, we have $\F_{S\cap\N^*}(\N^*)=\F_{S\cap\N^+}(\N^+)$; also, if $\N^+,\K^+\in\fN(\L^+)$, $\K^*=\K^+\cap\L^*$, $\N^*=\N^+\cap\L^*$ and $T=S\cap\N^+=S\cap\N^*$, then  $\K^+T=\N^+$ if and only if $\K^*T=\N^*$.

\smallskip

As noted above, every proper overgroup of an element of $\Delta^*\backslash\Delta=R^\F$ is in $\Delta$. Hence, by Lemma~\ref{L:ChermakIIHelp}(b),(c), properties (b) and (c) hold with $(\L^*,\Delta^*,S)$ in place of $(\L^+,\Delta^+,S)$. Suppose now that $\N^+$ is a partial normal subgroup of $\L^+$ and $\N:=\N^+\cap \L\unlhd\L$ such that $\F_{S\cap\N}(\N)$ is $\F$-invariant. Then $\N^*:=\N^+\cap\L^*\unlhd\L^*$ with $\N^*\cap\L=\N^+\cap\L=\N$. Since (b) is true with $(\L^*,\Delta^*,S)$ in place of $(\L^+,\Delta^+,S)$, it follows that $\F_{S\cap\N^*}(\N^*)=\F_{S\cap\N}(\N)$ and in particular, $\F_{S\cap\N^*}(\N^*)$ is $\F$-invariant. So $\F_{S\cap\N^+}(\N^+)=\F_{S\cap\N^*}(\N^*)=\F_{S\cap\N}(\N)$. This proves (b).

\smallskip

If $\N^+,\K^+\in\fN(\L^+)$ are arbitrary, $\N^*:=\L^*\cap\N^+$, $\K^*:=\L^*\cap\N^+$, $\N:=\L\cap\N^+$, $\K:=\L\cap\K^+$ and $T:=S\cap\N$, then we see similarly that
\[\K^+T=\N^+\Longleftrightarrow \K^*T=\N^*\Longleftrightarrow \K T=\N\]
and (c) holds. Hence, it remains to prove (a). 

\smallskip

If (a) is true with $(\L^*,\Delta^*,S)$ in place of $(\L^+,\Delta^+,S)$, then $\Phi_{\L^*,\L}$ is a bijection and $\Phi_{\L^*,\L}^{-1}$ is inclusion preserving. Hence, $\Phi_{\L^+,\L}=\Phi_{\L^*,\L}\circ \Phi_{\L^+,\L^*}$ is a bijection and $\Phi_{\L^+,\L}=\Phi_{\L^+,\L^*}^{-1}\circ \Phi_{\L^*,\L}^{-1}$ is inclusion preserving. Thus, replacing $(\L^+,\Delta^+,S)$ by $(\L^*,\Delta^*,S)$, we may assume from now on that 
\[\Delta^+=\Delta\cup R^\F.\]
Then in particular, every proper overgroup of an element of $\Delta^+\backslash\Delta=R^\F$ is an element of $\Delta$. So by  Lemma~\ref{L:LinkingLocN(R)}, $N_\L(R)=N_{\L^+}(R)$ is a subgroup of $\L$. 

\smallskip

Note that Lemma~\ref{L:ChermakIIHelp}(a) implies that $\Phi_{\L^+,\L}$ is injective. Moreover, if $\M^+$ and $\N^+$ are partial normal subgroups of $\L^+$ with $\M^+\cap\L\subseteq\N^+\cap\L$, then Lemma~\ref{L:ChermakIIHelp}(a) gives that $\M^+=\<(\M^+\cap\L)^{\L^+}\>\subseteq \N^+=\<(\N^+\cap\L)^{\L^+}\>$. So if $\Phi_{\L^+,\L}$ is a bijection, then $\Phi_{\L^+,\L}^{-1}$ is inclusion preserving. Hence, it remains to show that $\Phi_{\L^+,\L}$ is surjective. 

\smallskip

For the remainder of this proof let $\N$ be a partial normal subgroup of $\L$ and set $T:=S\cap\N$. We will show that there exists $\N^+\unlhd\L^+$ with $\N^+\cap\L=\N$. For the proof we set $\ov{\L}:=\L/\N$ and consider the natural projection 
\[\alpha\colon \L\longrightarrow \ov{\L}.\]
By \cite[Corollary~4.5]{loc1}, using the ``bar notation'', the triple $(\ov{\L},\ov{\Delta},\ov{S})$ is a locality. Observe also that $N_\L(R)\alpha\subseteq N_{\ov{\L}}(\ov{R})$.  We consider two cases now.

\smallskip

\textbf{Case~1: $T\not\leq R$.} As $T$ is strongly closed in $\F$ by \cite[Lemma~3.1]{loc1}, it follows then that $T\not\leq Q$ for every $Q\in R^\F$. Thus, for any such $Q$, we have $QT\in\Delta$ and $Q\alpha=\ov{Q}=\ov{QT}\in\ov{\Delta}$. This proves $\Delta^+\alpha\subseteq\ov{\Delta}$. Applying Corollary~\ref{C:MainCor} with $(\ov{\L},\ov{\Delta},\ov{S})$ in place of $(\tL,\tDelta,\tS)$, we conclude that there exists a homomorphism of partial group $\gamma\colon \L^+\longrightarrow \ov{\L}$ with $\gamma|_\L=\alpha$. By \cite[Lemma~3.3]{Ch}, $\N^+:=\ker(\gamma)$ is a partial normal subgroup of $\L^+$. Moreover, $\N^+\cap \L=\ker(\alpha)=\N$. 

\smallskip

\textbf{Case~2: $T\leq R$.} In this case, by Lemma~\ref{L:ProjectionLocalityMorphismFusionSystem}(a), $N_\L(R)\alpha=N_{\ov{\L}}(\ov{R})$. As $N_\L(R)=N_{\L^+}(R)$ is a subgroup of $\L$, it follows now from Lemma~\ref{L:PartialProj} that $M:=N_{\ov{\L}}(\ov{R})$ is a subgroup of $\ov{\L}$ and $\alpha|_{N_\L(R)}\colon N_\L(R)\longrightarrow M$ is a surjective group homomorphism. As $N_S(R)\in\Syl_p(N_\L(R))$ by Lemma~\ref{L:LocalityoverFFnatural}(b), this yields that $N_{\ov{S}}(\ov{R})=N_S(R)\alpha$ is a Sylow $p$-subgroup of $M$. Moreover, since $N_\F(R)=\F_{N_S(R)}(N_{\L^+}(R))=\F_{N_S(R)}(N_\L(R))$ by Lemma~\ref{L:LocalityoverFFnatural}(c), it follows from Lemma~\ref{L:GroupEpiFusionEpi} that $\alpha|_{N_S(R)}$ induces and epimorphism from $N_\F(R)$ to $\F_{N_{\ov{S}}(\ov{R})}(M)$. 

\smallskip

By Lemma~\ref{L:ProjectionLocalityMorphismFusionSystem}(b), $\alpha|_S$ induces an epimorphism from $\F=\F_S(\L)$ to $\ov{\F}:=\F_{\ov{S}}(\ov{\L})$. Hence, Lemma~\ref{L:FusionEpi} yields that $\Delta^+\alpha=\ov{\Delta}\cup \ov{R}^{\ov{\F}}$, the subgroup $\ov{R}$ is fully $\ov{\F}$-normalized, and $\alpha|_{N_S(R)}$ induces an epimorphism from $N_\F(R)$ to $N_{\ov{\F}}(\ov{R})$. The latter fact implies that $N_{\ov{\F}}(\ov{R})=\F_{N_{\ov{S}}(\ov{R})}(M)$. By Lemma~\ref{L:LinkingLocN(R)}, $R^*:=O_p(N_{\L^+}(R))\in\Delta$ and $R^*\unlhd N_\L(R)$. Hence, setting $\ov{\Delta}_{\ov{R}}=\{\ov{P}\in\ov{\Delta}\colon \ov{R}\unlhd\ov{P}\}$, we have $\ov{R^*}\in\ov{\Delta}_{\ov{R}}$ and $\ov{R^*}\unlhd M$, which implies  $\L_{\ov{\Delta}_{\ov{R}}}(M)=M=N_{\ov{\L}}(\ov{R})$. Now \cite[Hypothesis~5.3]{Ch} holds with $\ov{\F}$, $(\ov{\L},\ov{\Delta},\ov{S})$, $\ov{R}$ and $\id_M$ in place of $\F$, $(\L,\Delta,S)$, $T$ and $\lambda$. So by \cite[Theorem~5.14]{Ch}, setting $\tDelta:=\Delta^+\alpha$, there exists a locality $(\tL,\tDelta,\ov{S})$ such that $\ov{\L}\subseteq \tL$, $N_{\tL}(R)=M$, and the inclusion map $\ov{\L}\hookrightarrow\tL$ is a homomorphism of partial groups. Hence, $\alpha$ regarded as a map $\L\longrightarrow\tL$ is a homomorphism of partial groups, which by Corollary~\ref{C:MainCor} extends to a homomorphism $\gamma\colon \L^+\longrightarrow \tL$ of partial groups. Then $\N^+:=\ker(\gamma)\unlhd \L^+$ and $\N^+\cap\L=\ker(\alpha)=\N$. This proves the assertion.
\end{proof}

\begin{cor}\label{C:PartialNormal}
Let $(\L,\Delta,S)$ and $(\L^+,\Delta^+,S)$ be linking localities over the same fusion system $\F$. Then $|\fN(\L)|=|\fN(\L^+)|$.
\end{cor}

\begin{proof}
Suppose $(\L,\Delta,S)$ and $(\L^+,\Delta^+,S)$ are linking localities over the same fusion system $\F$. By Proposition~3.3 and Theorem~7.2(a) in \cite{subcentric}, there exist subcentric linking localities $(\hat{\L},\F^s,S)$ and $(\hat{\L}^+,\F^s,S)$ over $\F$ such that $\hat{\L}|_\Delta=\L$ and $\hat{\L}^+|_{\Delta^+}=\L^+$. Moreover, by \cite[Theorem~A(b)]{subcentric}, there exists a rigid isomorphism $\alpha\colon \hat{\L}\longrightarrow\hat{\L}^+$. Then $\alpha$ induces a bijection $\fN(\hat{\L})\longrightarrow \fN(\hat{\L}^+),\N\mapsto \N\alpha$. So by Theorem~\ref{T:MainChermakII} (applied twice), we have $|\fN(\L)|=|\fN(\hat{\L})|=|\fN(\hat{\L}^+)|=|\fN(\L^+)|$. This shows the assertion. 
\end{proof}

\bibliographystyle{amsalpha}
\bibliography{my.books}

\end{document}